\documentclass[11pt,a4paper]{article}

\usepackage{titlesec}
\usepackage{fancyhdr}
\usepackage{a4wide}
\usepackage{graphicx}
\usepackage{float}
\usepackage{amssymb}
\usepackage{amsmath}
\usepackage{amsthm}
\usepackage{color}
\usepackage{mathrsfs}
\usepackage{array}
\usepackage{eucal}
\usepackage{tikz}
\usepackage[T1]{fontenc}
\usepackage{inputenc}
\usepackage[english]{babel}
\usepackage{frcursive}
\usepackage{lmodern}
\usepackage{hyperref}
\usepackage{geometry}
\usepackage{changepage}
\changepage{0pt}{}{}{}{}{0pt}{}{0pt}{5pt}
\usepackage[numbers]{natbib}
\usepackage{hyperref}
\hypersetup{
pdfpagemode=none,
pdftoolbar=true,        
pdfmenubar=true,        
pdffitwindow=false,     
pdfstartview={Fit},    
pdftitle={Perfect transmission invisibility for waveguides with sound hard walls},    
pdfauthor={A.-S. Bonnet-Ben Dhia, L. Chesnel, S.A. Nazarov},     
pdfsubject={},  
pdfcreator={A.-S. Bonnet-Ben Dhia, L. Chesnel, S.A. Nazarov},   
pdfproducer={A.-S. Bonnet-Ben Dhia, L. Chesnel, S.A. Nazarov}, 
pdfkeywords={}, 
pdfnewwindow=true,      
colorlinks=true,       
linkcolor=magenta,          
citecolor=red,        
filecolor=cyan,      
urlcolor=blue           
}

\newcommand{\dsp}{\displaystyle}
\newcommand{\eps}{\varepsilon}
\newcommand{\om}{\omega}
\newcommand{\Om}{\Omega}
\newcommand{\mrm}[1]{\mathrm{#1}}
\newcommand{\Cplx}{\mathbb{C}}
\newcommand{\N}{\mathbb{N}}
\newcommand{\R}{\mathbb{R}}

\newcommand{\mH}{\mrm{H}}

\newcommand{\mX}{\mrm{X}}
\newcommand{\mY}{\mrm{Y}}
\newcommand{\mZ}{\mrm{Z}}
\newcommand{\loc}{\mbox{\scriptsize loc}}

\newtheorem{lemma}{Lemma}[section]
\newtheorem{remark}{Remark}[section]

\newtheorem{proposition}{Proposition}[section]

\begin{document}

~\vspace{0.3cm}
\begin{center}
{\sc \bf\LARGE 
Perfect transmission invisibility for waveguides \\[7pt] with sound hard walls}
\end{center}

\begin{center}
\textsc{Anne-Sophie Bonnet-Ben Dhia}$^1$, \textsc{Lucas Chesnel}$^2$, \textsc{Sergei A. Nazarov}$^{3,\,4,\,5}$\\[16pt]
\begin{minipage}{0.9\textwidth}
{\small
$^1$ Laboratoire Poems, CNRS/ENSTA/INRIA, Ensta ParisTech, Universit\'e Paris-Saclay, 828, Boulevard des Mar\'echaux, 91762 Palaiseau, France; \\
$^2$ INRIA/Centre de math\'ematiques appliqu\'ees, \'Ecole Polytechnique, Universit\'e Paris-Saclay, Route de Saclay, 91128 Palaiseau, France;\\
$^3$ St. Petersburg State University, Universitetskaya naberezhnaya, 7-9, 199034, St. Petersburg, Russia;\\
$^4$ Peter the Great St. Petersburg Polytechnic University, Polytekhnicheskaya ul, 29, 195251, St. Petersburg, Russia;\\
$^5$ Institute of Problems of Mechanical Engineering, Bolshoy prospekt, 61, 199178, V.O., St. Petersburg, Russia.\\[10pt]
E-mails: \texttt{Anne-Sophie.Bonnet-Bendhia@ensta-paristech.fr}, \texttt{lucas.chesnel@inria.fr}, \texttt{srgnazarov@yahoo.co.uk}\\[-14pt]
\begin{center}
(\today)
\end{center}
}
\end{minipage}
\end{center}
\vspace{0.4cm}

\noindent\textbf{Abstract.} 
We are interested in a time harmonic acoustic problem in a waveguide with locally perturbed sound hard walls. We consider a setting where an observer generates incident plane waves at $-\infty$ and probes the resulting scattered field at $-\infty$ and $+\infty$. Practically, this is equivalent to measure the reflection and transmission coefficients respectively denoted $R$ and $T$. In \cite{BoNa13}, a technique has been proposed to construct waveguides with smooth walls such that $R=0$ and $|T|=1$ (non reflection). However the approach fails to ensure $T=1$ (perfect transmission without phase shift). In this work, first we establish a result explaining this observation. More precisely, we prove that for wavenumbers smaller than a given bound $k_{\star}$ depending on the geometry, we cannot have $T=1$ so that the observer can detect the presence of the defect if he/she is able to measure the phase at $+\infty$. In particular, if the perturbation is smooth and small (in amplitude and in width), $k_{\star}$ is very close to the threshold wavenumber. Then, in a second step, we change the point of view and, for a given wavenumber, working with singular perturbations of the domain, we show how to obtain $T=1$. In this case, the scattered field is exponentially decaying both at $-\infty$ and $+\infty$. We implement numerically the method to provide examples of such undetectable defects.\\


\section{Introduction}\label{Introduction}

Invisibility is a very dynamic field of research in inverse scattering theory. In this article, we are interested in a time harmonic acoustic problem in a waveguide with a bounded transverse section. There is an important literature concerning techniques of imaging for waveguides (see e.g \cite{RoFi00,PAHW07,BoLu08,ArGL11,MoSK12,BoFl14,BoNg15}). Here, we consider a situation where an observer wants to detect the presence of defects in some reference waveguide from far-field data. We assume that the observer is located far from the defect. Practically, the observer generates waves, say from $-\infty$, and measures the amplitude of the resulting scattered field at $\pm\infty$. It is known that, at a given frequency, at $\pm\infty$, the scattered field decomposes as the sum of a finite number of propagative waves plus some exponentially decaying remainder. In this work, we will assume that the frequency is sufficiently small so that only one wave (the piston mode for the problem considered here with sound hard walls) can propagate in the waveguide. In this case, one usually introduces two complex coefficients, namely the \textit{reflection} and \textit{transmission} coefficients, denoted $R$ and $T$, such that $R$ (resp. $T-1$) corresponds to the amplitude of the scattered field at $-\infty$ (resp. $+\infty$) (see (\ref{DecompoChampTotal})). According to the energy conservation, we have
\begin{equation}\label{EnergyConservationOriginial}
|R|^2+|T|^2=1.
\end{equation} 
We shall say that the defects in the reference waveguide are perfectly invisible at a given frequency if there holds $R=0$ and $T=1$. In such a situation, the scattered field is exponentially decaying at $\pm\infty$ and the observer cannot detect the presence of the defect from noisy measurements.\\ 
\newline
In this context, examples of quasi invisible obstacles ($|R|$ small or $|T-1|$ small), obtained via numerical simulations, exist in literature. We refer the reader to \cite{EvMP14} for a water waves problem and to \cite{AlSE08,EASE09,NgCH10,OuMP13,FuXC14} for strategies based on the use of new ``zero-index'' and ``epsilon near zero'' metamaterials in electromagnetism (see \cite{FlAl13} for an application to acoustic). Let us mention also that the problem of the existence of quasi invisible obstacles for frequencies close to the threshold frequency has been addressed in the analysis of the so-called Weinstein anomalies \cite{Vain66} (see e.g.  \cite{na580,KoNS16}).
\\
\newline
Recently in \cite{BoNa13,BLMN17}, (see also \cite{BoNTSu,BoCNSu,ChHS15,ChNa16} for applications to other problems) a technique has been proposed to prove the existence of waveguides different from the straight geometry $\R\times\om$, where $\om$ is the cross section, with sound hard walls such that $R=0$ (rigorously). If an observer located at $-\infty$ generates a propagative wave in such a waveguide, then the scattered field is exponentially decaying at $-\infty$. Therefore, the perturbation of the walls is invisible for backscattering measurements. These waveguides are said to be \textit{non-reflecting}. Now, imagine that the observer located at $-\infty$ can also measure the transmission coefficient $T$. As a consequence of formula (\ref{EnergyConservationOriginial}), when $R=0$,  we have $|T|=1$. In this case, if the observer measures only the amplitude of waves (the modulus of $T$), again the perturbation is invisible. But if he/she can measure the phase of waves (and so, the complex value of $T$), he/she is able to detect the presence of the perturbation when $|T|=1$ and $T\ne 1$. At this point, a natural question is: can we design a defect such that $R=0$ and $T=1$?\\
\newline
Before proceeding further, let us explain briefly the idea of the technique used in \cite{BoNa13} to construct a waveguide $\Om$ such that $R=0$. The method consists in adapting the proof of the implicit function theorem. It has been introduced in \cite{Naza11,Naza11c,Naza12,Naza13,CaNR12,Naza11b}. In these works, the authors construct small regular and singular perturbations of the walls of a waveguide that preserve the multiplicity of the point spectrum on a given interval of the continuous spectrum. For our problem, assume that the geometry of $\Om=\Om(\sigma)$ is defined by a parameter $\sigma$ corresponding to the perturbation of the walls and belonging to a space $\mX$ of smooth and compactly supported functions. We denote $R(\sigma)$ (resp. $T(\sigma)$) the reflection (resp. transmission) coefficient in $\Om(\sigma)$. First, we note that in the reference waveguide, the scattered field associated with an incident wave is null. Therefore, we have $R(0)=0$ and $T(0)=1$. Our goal is to find some $\sigma\not\equiv0$ such that $R(\sigma)=0$. Looking for small perturbations of $\Om(0)$ allows to compute an asymptotic (Taylor) expansion of $R(\sigma)$:
\[
R(\sigma)=R(0)+dR(0)(\sigma)+O(\|\sigma\|^2)=dR(0)(\sigma)+O(\|\sigma\|^2),
\] 
where $dR(0)$, the differential of $R$ at zero, is a linear map from $\mX$ to $\Cplx$. It is easy to find $\sigma_0\in\mX$ such that 
\begin{equation}\label{syshomo}
dR(0)(\sigma_0)=0
\end{equation}  
because (\ref{syshomo}) is a system of two equations (real and imaginary parts) and $\sigma_0$ is a function. A perturbation of the form $\sigma=\eps\sigma_0$ is hard to detect for small $\eps$ because $R(\eps\sigma_0)=O(\eps^2)$. However, it is not perfectly invisible. To impose $R(\sigma)=0$, we look for $\sigma$ under the form $\sigma=\eps(\sigma_0+\tau_{r}\sigma_{r}+\tau_{i}\sigma_{i})$ where $\tau_{r}$, $\tau_{i}$ are some real parameters to tune and where $\sigma_{r}$, $\sigma_i\in\mX$ are such that $dR(0)(\sigma_{r})=1$, $dR(0)(\sigma_{i})=i$ (the latter $i$ is the usual complex number such that $i^2=-1$). Remark that $\sigma_{r}$, $\sigma_i\in\mX$ exist if and only if $dR(0):\mX\to\Cplx$ is onto. If the latter assumption is true, we can impose $R(\sigma)=0$ solving the fixed point problem
\begin{equation}\label{PbFixedPointIntro}
\begin{array}{|l}
\mbox{Find }\tau=(\tau_{r},\tau_{i})^{\top}\in\R^{2}\mbox{ such that }\tau=-(\Re e\,\tilde{R}^{\eps}(\tau),\Im m\,\tilde{R}^{\eps}(\tau))^{\top},
\end{array}~\\[5pt]
\end{equation}
where $\tilde{R}^{\eps}(\tau)\in\Cplx$ is the abstract remainder such that $R(\sigma)=dR(0)(\sigma)+\tilde{R}^{\eps}(\tau)$. Here $\Re e$ and $\Im m$ stand respectively for the real and imaginary parts. For $\eps$ small enough, we can prove that Problem (\ref{PbFixedPointIntro}) admits a  solution. To summarize, this approach allows to construct a (non-trivial) perturbation such that $R(\sigma)=0$ as soon as $dR(0):\mX\to\Cplx$ is onto.\\
\newline
We would like to use the same method to find perturbations such that $T(\sigma)=1$. However, a calculation shows that $dT(0)(\sigma)$ is null for all $\sigma\in\mX$ (see formula (29) in \cite{BoNa13}). As a consequence, $dT(0):\mX\to\Cplx$ is not onto and the technique described above does not apply straightforwardly.\\
\newline
The methodology proposed in \cite{BoNa13} has been adapted in \cite{BoCNSu} to inverse obstacle scattering theory in free space (the  waveguide with a bounded transverse section is replaced by $\R^d$, $d\ge2$). In the latter article, it is explained how to construct defects in a reference medium (defect in the coefficient representing the physical properties of the material instead of defect in the geometry) which are invisible to a finite number of far field measurements. Interestingly, the technique fails when among the directions of observation, there is the incident direction. The reason is the same as above: in this case the differential of the far field pattern with respect to the material perturbation is not onto in $\Cplx$. In \cite{BoCNSu}, it is also shown an additional result, known as the \textit{optical theorem} (see Lord Rayleigh \cite{Rayl71} and \cite[\S10.11]{Jack99},\cite{Newt76,Mans12}). It is proved that imposing far field invisibility in the incident direction requires to impose far field invisibility in all directions. This is very restricting and probably impossible to obtain (see \cite{BlPS14,PaSV14,ElHu15,HuSV16} for related works). At this stage, it seems that configurations where the technique of  \cite{BoCNSu} fails correspond to situations where there is an intrinsic obstruction to invisibility. In the present work, we wish to study whether or not such an intrinsic obstruction to invisibility holds for the waveguide problem with sound hard walls. Is it impossible, like for the problem in free space, to impose $T=1$ or is it just our technique based on the proof of the implicit function theorem which is inefficient? We wish to point out that obtaining $T=1$ can be easily realized in  waveguides with sound soft walls. The reason is that for such problems, one finds that $dT(0)$ is not the null mapping and using the technique described above, \textit{i.e.} playing with small and smooth perturbations, one can construct geometries where $R=0$ and $T=1$.\\ 
\newline
The outline is as follows. In Section \ref{SectionSetting}, we first introduce the notation that will be used throughout the article. Then, in Section \ref{SectionObstruction}, we prove that for a given non trivial waveguide different from the reference setting, there is a $k_{\star}$, depending on the geometry, such that for $k\le k_{\star}$, the transmission coefficient $T$ satisfies $T\ne1$ (Proposition \ref{MainProposition}). In Section \ref{SectionConstruction}, for a given wavenumber, we explain how to construct invisible perturbations such that $T=1$ (Proposition \ref{propositionMainResult}). In Section \ref{SectionNumericalExperiments}, we implement numerically the method to provide examples of such undetectable defects. Section \ref{SectionJustifAsym} is devoted to the justification of intermediate results of the asymptotic procedure allowing to make the analysis rigorous. Finally, in Section \ref{SectionConclu} we give a brief conclusion and describe some routes to investigate as well as open questions. \\
\newline
Let us emphasize that Section \ref{SectionObstruction} and Sections \ref{SectionConstruction},\ref{SectionNumericalExperiments},\ref{SectionJustifAsym} present two different approaches.  They can be read independently. In Section \ref{SectionObstruction}, there is no assumption on the size of the geometrical perturbation. On the other hand, in Sections \ref{SectionConstruction},\ref{SectionNumericalExperiments},\ref{SectionJustifAsym}, we work with chimneys of width $\eps$ small compared to the wavelength. The main results of this work are Proposition \ref{MainProposition} and Proposition \ref{propositionMainResult}.

\section{Setting}\label{SectionSetting}

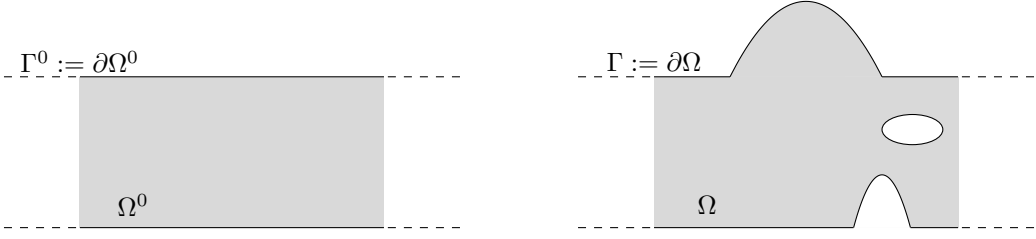
\begin{figure}[!ht]
\centering
\begin{tikzpicture}
\draw[fill=gray!30,draw=none](-2,0) rectangle (2,2);
\draw (-2,0)--(2,0);
\draw (-2,2)--(2,2);
\draw [dashed](-3,0)--(-2,0);
\draw [dashed](3,0)--(2,0);
\draw [dashed](-3,2)--(-2,2);
\draw [dashed](3,2)--(2,2);
\node at (-1.3,0.3){\small$\Omega^0$};
\node at (-2,2.2){\small$\Gamma^0:=\partial\Omega^0$};
\end{tikzpicture}\qquad\qquad
\begin{tikzpicture}
\draw[fill=gray!30,draw=none](-2,0) rectangle (2,2);
\draw (-2,0)--(0.62583426132,0);
\draw (2,0)--(1.37416573868,0);
\draw [white](0.62583426132,0)--(1.37416573868,0);
\draw (-2,2)--(-1,2);
\draw (2,2)--(1,2);
\draw[gray!30](-1,2)--(1,2);
\draw[fill=gray!30,domain=-1:1,smooth,variable=\x] plot ({\x},{-\x*\x+3});
\draw[fill=white,domain=0.62583426132:1.37416573868,smooth,variable=\x] plot ({\x},{-5*(\x-1)*(\x-1)+0.7});
\draw [dashed](-3,0)--(-2,0);
\draw [dashed](3,0)--(2,0);
\draw [dashed](-3,2)--(-2,2);
\draw [dashed](3,2)--(2,2);
\draw[fill=white] (1.4,1.3) ellipse (0.4 and 0.2);
\node at (-1.3,0.3){\small$\Omega$};
\node at (-2,2.2){\small$\Gamma:=\partial\Omega$};
\end{tikzpicture}
\caption{Examples of reference (left) and perturbed (right) waveguides. \label{DomainPerturbation}} 
\end{figure}

\noindent We are interested in the propagation of acoustic waves in time harmonic regime in a waveguide with sound hard walls containing a defect (perturbation of the walls and/or presence of a sound hard obstacle). To model such a problem, introduce $\Om^{0}:=\{z=(x,y)\,|\,x\in\R\mbox{ and }y\in\om\}$ a \textit{reference} waveguide of $\R^d$, $d\ge2$. The cross-section $\om\subset \R^{d-1}$ is a connected open set whose boundary is Lipschitz continuous. Then, introduce $\Om$ a \textit{perturbed} waveguide which coincides with $\Om^{0}$ for $z=(x,y)$ such that $|x|\ge L/2$ (see Figure \ref{DomainPerturbation} for an illustration). Here, $L>0$ is a fixed number. We assume that $\Om$ is connected and that its boundary $\Gamma:=\partial\Om$ is Lipschitz continuous. We denote $\nu$ (resp. $\nu^0$) the normal unit vector to $\Gamma$ (resp. $\Gamma^0:=\partial\Om^0$) directed to the exterior of $\Om$ (resp. $\Om^0$). To model the propagation of waves in $\Om$, we consider the problem
\begin{equation}\label{pbInit}
\begin{array}{|rcll}
-\Delta u  & = & k^2 u & \mbox{ in }\Om\\
\partial_{\nu} u  & = & 0  & \mbox{ on }\Gamma.
\end{array}
\end{equation}
In (\ref{pbInit}), $u$ represents the pressure in the medium filling the waveguide, $k$ corresponds to the wavenumber proportional to the frequency of harmonic oscillations, $\Delta$ is the Laplace operator and $\partial_{\nu} =\nu\cdot\nabla$ denotes the derivative along the outward normal defined almost everywhere on the Lipschitz boundary. Using separation of variables in the reference waveguide $\Om^0$, we can compute the solutions of the problem
\begin{equation}\label{pbInitRef}
\begin{array}{|rcll}
-\Delta u  & = & k^2 u & \mbox{ in }\Om^0\\
\partial_{\nu^0} u  & = & 0  & \mbox{ on }\Gamma^0.
\end{array}
\end{equation}
To provide their expression, let us introduce $\lambda_n$ and $\varphi_n$, the eigenvalues and the corresponding eigenfunctions of the Neumann problem for the Laplace operator on the cross-section $\om$, such  that
\begin{equation}\label{eigenpairTransverse}
\begin{array}{l}
0=\lambda_0<\lambda_1 \le \lambda_2 \le \lambda_3 \le \cdots \le \lambda_n \le \cdots \rightarrow +\infty,\\
\varphi_n\in\mH^1(\om),\qquad(\varphi_m,\varphi_n)_{\om}=\delta_{m,n},\qquad m,n\in\N:=\{0,1,2,\dots\}. 
\end{array}
\end{equation}
Here, $\delta_{m,n}$ stands for the Kronecker symbol and $(\cdot,\cdot)_{\om}$ is the inner product of $\mrm{L}^2(\om)$. In particular, note that $\varphi_0=|\om|^{-1/2}$ where $|\om|=\mrm{meas}_2(\om)$. Assume that $k\in\R$ is such that $k^2\ne \lambda_n$ for all $n\in\N$. We call \textit{modes} of the waveguide the solutions of (\ref{pbInitRef}) given by
\begin{equation}\label{defModes}
w^{\pm}_n(x,y)= (2|\beta_n|)^{-1/2} e^{\pm i \beta_n x}\varphi_n(y)\quad \mbox{ with }\ \beta_n:=\sqrt{k^2-\lambda_n}.
\end{equation}
The complex square root is chosen so that if $\xi=r e^{i\gamma}$ for $r\ge 0$ and $ \gamma \in[0;2\pi)$, then $\sqrt{\xi}=\sqrt{r}e^{i\gamma/2}$. With this choice, there holds $\Im m\,\sqrt{\xi}\ge 0$ for all $\xi\in\Cplx$. The normalization coefficients in (\ref{defModes}) are set to simplify some formulas below (see e.g. (\ref{expressionReflecCoeff})). According to the value of $k^2$ with respect to the $\lambda_n$, the modes $w^{\pm}_n$ adopt different behaviours. In the present work, we shall assume that the wavenumber $k$ verifies 
\begin{equation}\label{AssumptionWaveNber}
0 < k^2 < \lambda_1.
\end{equation}
In this case, for $n\ge 1$, the function $w^{+}_n$ (resp. $w^{-}_n$) decays exponentially at $+\infty$ (resp. $-\infty$) and grows exponentially at $-\infty$ (resp. $+\infty$). On the other hand, the functions $w^{\pm}_0$ are \textit{oscillating waves} in $\Om^0$. To shorten notation, we  denote $w^{\pm}:=w^{\pm}_0$. In $\Om$, the waves $w^{\pm}$ travel from $\mp\infty$, in the positive/negative direction of the axis $(Ox)$ and are scattered by the defect of the waveguide.\\
\newline
Let us denote $\mrm{H}^1_{\loc}(\Om)$ the set of measurable functions whose $\mH^1$-norm is finite on each bounded subset of $\Om$. We will say that a function $v\in\mrm{H}^1_{\loc}(\Om)$ which satisfies equations (\ref{pbInit}) is \textit{outgoing} if it admits the decomposition
\begin{equation}\label{scatteredField}
v = \chi^+ s^{+} w^+ + \chi^- s^{-} w^- + \tilde{v}\ ,
\end{equation}
for some constants $s^{\pm}\in\Cplx$. In (\ref{scatteredField}), $\tilde{v}\in\mH^1(\Om)$ denotes a remainder which is exponentially decaying at $\pm\infty$ while $\chi^{+}\in\mathscr{C}^{\infty}(\Om^0)$ (resp. $\chi^{-}\in\mathscr{C}^{\infty}(\Om^0)$) is a cut-off function that is equal to one for $x\ge L$ (resp. $x\le -L$) and equal to zero for $x\le L/2$ (resp. $x\ge -L/2$). We remind the reader that the constant $L$ is chosen so that $\Om$ coincides with $\Om^{0}$ for $z=(x,y)$ such that $|x|\ge L/2$. We emphasize that $L$ plays absolutely no role in the following. In particular, the results are independent from $L$. Now, the scattering problem we consider states 
\begin{equation}\label{PbChampTotalBIS}
\begin{array}{|rcll}
\multicolumn{4}{|l}{\mbox{Find }u\in\mrm{H}^1_{\loc}(\Om) \mbox{ such that }u-w^+\mbox{ is outgoing and } }\\[3pt]
-\Delta u  & = & k^2 u & \mbox{ in }\Om\\[3pt]
\partial_{\nu} u  & = & 0  & \mbox{ on }\Gamma.\end{array}
\end{equation}
Classically, for all $k^2\in(0;\lambda_1)$, one proves that if $u$ is a solution to (\ref{PbChampTotalBIS}), then we have the decomposition 
\begin{equation}\label{DecompoChampTotal}
u-\chi^-w^+ = \chi^+ T w^+ + \chi^- R w^- + \tilde{u},
\end{equation}
where $\tilde{u}\in\mH^1(\Om)$ is a term which is exponentially decaying at $\pm\infty$. In (\ref{DecompoChampTotal}), $R$ is the reflection coefficient and  $T$ is the transmission coefficient.\\ 
\newline
In the following, $u_s:=u-w^+$ (resp. $u$) will be referred to as the \textit{scattered} (resp. \textit{total}) field associated with the \textit{incident} field $u_i:=w^+$. Note that $u_i=w^+=(2k)^{-1/2}e^{ikx}\varphi_0$ can be naturally extended to $\Om\setminus\overline{\Om^0}$ and that we have $\Delta u_i+k^2u_i=0$ in $\Om$. To simplify notation,  we introduce the coefficients $s^{\pm}$ such that
\[
s^{-} = R\qquad \mbox{ and }\qquad s^{+} = T-1.
\]
With this definition, according to (\ref{DecompoChampTotal}), we have $u_s= \chi^+ s^+ w^+ + \chi^- s^- w^- + \hat{u}$ for some $\hat{u}\in\mH^1(\Om)$ which is exponentially decaying at $\pm\infty$.

\section{Obstruction to transmission invisibility}\label{SectionObstruction}

In this section, we prove that for a given geometry of the waveguide $\Om$, for wavenumbers small enough, the transmission coefficient in the decomposition (\ref{DecompoChampTotal}) cannot be equal to one. The main result is Proposition  \ref{MainProposition}. First we show the following identity.

\begin{lemma}\label{lemmaPrincipal}
If $u$ satisfies Problem (\ref{PbChampTotalBIS}), then the scattered field $u_s=u-u_i$ satisfies the identity
\begin{equation}\label{DesiredIdentity}
\Im m\,s^+ = \int_{\Om}|\nabla u_s|^2 - k^2|u_s|^2\,dz.
\end{equation}
\end{lemma}
\begin{remark}\label{remarkDef}
Note that the integral on the right hand side of (\ref{DesiredIdentity}) is well-defined. To show this, use the decomposition $u_s= \chi^+ s^+ w^+ + \chi^- s^- w^- + \hat{u}$, where $\hat{u}$ stands for a function which is exponentially decaying at $\pm\infty$, and  notice that there holds $|\nabla w^+|^2 - k^2|w^+|^2=|\nabla w^-|^2 - k^2|w^-|^2=0$ in $\Om$ because $w^{\pm} = (2k)^{-1/2} e^{\pm i k x}\varphi_0$ where $\varphi_0$ is a constant.
\end{remark}
\begin{proof}
For $\ell>L$, define $\Om_\ell:=\{z=(x,y)\in\Om\,|\,-\ell<x<\ell \}$. From the equation $\Delta u_s+k^2 u_s=0$, multiplying by $\overline{u_s}$ and integrating by parts, we find 
\begin{equation}\label{decompo0}
\int_{\Om_\ell}|\nabla u_s|^2 - k^2|u_s|^2\,dz-\int_{\partial\Om\cap\partial\Om_\ell}\partial_{\nu} u_s \overline{u_s}\,d\sigma  -\int_{\Sigma_{-\ell}\cup\Sigma_{+\ell}}\partial_{\nu} u_s \overline{u_s}\,d\sigma= 0. 
\end{equation}
Here, we denote $\Sigma_{\pm\ell}:=\{\pm\ell\}\times\om$ and $\partial_{\nu}=\pm \partial_x$ at $x=\pm\ell$. On $\partial\Om\cap\partial\Om_\ell$, we have $\partial_{\nu} u=0$ which implies  $\partial_{\nu} u_s=-\partial_{\nu} u_i$. Since $\overline{u_s}=\overline{u}-\overline{u_i}$, we obtain
\begin{equation}\label{decompo1}
-\dsp\int_{\partial\Om\cap\partial\Om_\ell}\partial_{\nu} u_s \overline{u_s}\,d\sigma =  \dsp\int_{\partial\Om\cap\partial\Om_\ell}\partial_{\nu} u_i \overline{u}-\partial_{\nu} u_i \overline{u_i}\,d\sigma.
\end{equation}
To deal with the first term in the integrand on the right hand side of (\ref{decompo1}), we can write
\begin{equation}\label{decompo2}
\begin{array}{lcl}
\dsp\int_{\partial\Om\cap\partial\Om_\ell}\partial_{\nu} u_i \overline{u}\,d\sigma & = & \dsp\int_{\partial\Om\cap\partial\Om_\ell}\partial_{\nu} u_i \overline{u}- u_i \partial_{\nu}\overline{u}\,d\sigma \\[10pt]
& = & -\dsp\int_{\Sigma_{-\ell}\cup\Sigma_{+\ell}}\partial_{\nu} u_i \overline{u}-u_i \partial_{\nu}\overline{u}\,d\sigma.
\end{array}
\end{equation}
On the other hand, for the second term of the integrand on the right hand side of (\ref{decompo1}), integrating by parts, we get
\begin{equation}\label{decompo3}
\begin{array}{lcl}
-\dsp\int_{\partial\Om\cap\partial\Om_\ell}\partial_{\nu} u_i \overline{u_i}\,d\sigma & = & \dsp\int_{(\Om\setminus\overline{\Om^0})\cup(\Om^0\setminus\overline{\Om})} |\nabla u_i|^2+\Delta u_i \overline{u_i}\,dz \\[10pt]
 & = & \dsp\int_{(\Om\setminus\overline{\Om^0})\cup(\Om^0\setminus\overline{\Om})} |\nabla u_i|^2-k^2 |u_i|^2\,dz \ =\ 0.
\end{array}
\end{equation}
The last equality in (\ref{decompo3}) has been obtained by using again the expression $u_i=w^+ = (2k)^{-1/2} e^{i k x}\varphi_0$. Plugging (\ref{decompo1}) in (\ref{decompo0}) and using (\ref{decompo2}), (\ref{decompo3}) yields  
\begin{equation}\label{decompo4}
\int_{\Om_\ell}|\nabla u_s|^2 - k^2|u_s|^2\,dz=\dsp\int_{\Sigma_{-\ell}\cup\Sigma_{+\ell}}\partial_{\nu} u_i \overline{u}- u_i \partial_{\nu}\overline{u}+\partial_{\nu}u_s \overline{u_s}\,d\sigma. 
\end{equation}
From the decomposition $u_s=u-u_i= \chi^+ s^+ w^+ + \chi^- s^- w^- + \hat{u}$, where $\hat{u}$ is exponentially decaying at $\pm\infty$, we find
\begin{equation}\label{limit1}
\lim_{\ell\to+\infty}\dsp\int_{\Sigma_{+\ell}}\partial_{\nu} u_i \overline{u}-u_i \partial_{\nu}\overline{u}+\partial_{\nu} u_s \overline{u_s}\,d\sigma = i(1+\overline{s^+}+\frac{1}{2}\,|s^+|^2)
\end{equation}
and 
\begin{equation}\label{limit2}
\lim_{\ell\to+\infty}\dsp\int_{\Sigma_{-\ell}}\partial_{\nu} u_i \overline{u}-u_i \partial_{\nu}\overline{u}+\partial_{\nu} u_s \overline{u_s}\,d\sigma = i(-1+\frac{1}{2}\,|s^-|^2).
\end{equation}
Taking the limit as $\ell\to+\infty$ in (\ref{decompo4}) and using (\ref{limit1}), (\ref{limit2}) leads to 
\begin{equation}\label{Finalequality}
\int_{\Om}|\nabla u_s|^2 - k^2|u_s|^2\,dz= i(\overline{s^+}+\frac{1}{2}\,(|s^+|^2+|s^-|^2)). 
\end{equation}
Extracting the real part of (\ref{Finalequality}) gives the desired identity (\ref{DesiredIdentity}). 
\end{proof}
\begin{proposition}\label{MainProposition}
There is some $k_{\star}\in(0;\sqrt{\lambda_1})$ such that for all $k\in(0;k_{\star}]$, the transmission coefficient $T$ appearing in the decomposition (\ref{DecompoChampTotal}) of $u$, a solution to Problem (\ref{PbChampTotalBIS}), satisfies $T\ne 1$. 
\end{proposition}
\begin{remark}\label{RmkCrack}
We assumed that $\Om$ has a Lipschitz continuous boundary. This condition can be weakened. In our proof (see in particular the penultimate line), we just need that the classical Green formula holds in $\Om$ and that $\partial\Om$ contains a set of non zero measure where $\nu_x\ne0$ (we use the notation $\nu=(\nu_x,\nu_y)^{\top}\in\R\times\R^{d-1}$). Observe that the latter assumption on $\partial\Om$ excludes the case where $\Om=\Om^0\setminus\overline{\mathcal{F}}$, $\mathcal{F}$ being a family of cracks parallel to the $(Ox)$ axis. This is reassuring because in this setting the result of Proposition \ref{MainProposition} clearly does not hold. Indeed, for all wavenumbers
the incident field $u_i$ goes through the waveguide without producing any scattered field ($u_s=0$ so that $T=1$). 
\end{remark}
\begin{proof}
Assume that $T$ satisfies $T=1\Leftrightarrow s^+=0$. In this case, the conservation of energy $1=|T|^2+|R|^2$ imposes $R=s^-=0$. Lemma \ref{lemmaPrincipal} above implies
\begin{equation}\label{NullEnergy}
0=\int_{\Om}|\nabla u_s|^2 - k^2|u_s|^2\,dz=\int_{\Om\setminus\overline{\Om_b}}|\nabla u_s|^2 - k^2|u_s|^2\,dz+\int_{\Om_b}|\nabla u_s|^2 - k^2|u_s|^2\,dz.
\end{equation}
Here $\Om_b=\{z=(x,y)\in\Om\,|\,-x_-<x<x_+ \}$ where $x_\pm$ are chosen so that $\Om$ coincides with $\Om^{0}$ for $z=(x,y)$ such that $\pm x\ge x_{\pm}$. Now, we estimate each of the two terms on the right hand side of (\ref{NullEnergy}).\\
\newline
$\star$ Let us consider the first one. For $\pm x> x_{\pm}$, classically, we can decompose $u_s$ as 
\[
u_s(x,y) = s^{\pm}w^{\pm}(x,y)+\sum_{n=1}^{+\infty} \alpha^{\pm}_n w^{\pm}_n(x,y).
\]
In this expression, according to (\ref{defModes}), we have 
\[
w^{\pm}(x,y) = \frac{1}{\sqrt{2k}}\,e^{\pm i k x}\varphi_0(y)\qquad \mbox{ and, for $n\ge1$, }\qquad  w^{\pm}_n(x,y) = \frac{1}{\sqrt{2|\beta_n|}}\,e^{\mp |\beta_n| x}\varphi_n(y),
\]
Moreover, there holds 
\begin{equation}\label{defCoeff}
s^{\pm} = \sqrt{2k} e^{-ikx_{\pm}}\int_{\{\pm x_{\pm}\}\times\om} u_s\varphi_0\,d\sigma\quad \mbox{and, for $n\ge1$,}\quad\alpha^{\pm}_n=\sqrt{2|\beta_n|} e^{|\beta_n|x_{\pm}}\int_{\{\pm x_{\pm}\}\times\om} u_s\varphi_n\,d\sigma.
\end{equation}
Define the semi-infinite cylinders $\Om_+=(x_+;+\infty)\times\om$ and $\Om_{-}=(-\infty;-x_-)\times\om$. A direct calculation using the orthonormality of the family $(\varphi_n)_{n\ge0}$ yields
\begin{equation}\label{estim1}
\begin{array}{lcl}
\dsp\int_{\Om\setminus \overline{\Om_b}} |\nabla u_s|^2 - k^2|u_s|^2\,dz & = & \dsp\sum_{n=1}^{+\infty} 2(\lambda_n-k^2)\Big(|\alpha^{+}_n|^2 \dsp\int_{\Om_+}|w^{+}_n|^2\,dz+|\alpha^{-}_n|^2 \dsp\int_{\Om_-}|w^{-}_n|^2\,dz\Big)\\[10pt]
& \ge & 2(\lambda_1-k^2)\Big( \dsp\int_{\Om_+}|u_s|^2\,dz+\dsp\int_{\Om_-}|u_s|^2\,dz\Big). 
\end{array}
\end{equation}
$\star$ Now, we deal with the second term on the right hand side of (\ref{NullEnergy}). Above, we have seen that $s^{+}=0$ implies $s^{-}=0$. Then, according to (\ref{defCoeff}), there holds $\int_{\{x_{+}\}\times\om} u_s\,d\sigma=\int_{\{-x_{-}\}\times\om} u_s\,d\sigma=0$. Define the Hilbert space $\mY:=\{\psi\in\mH^1(\Om_b)\,|\,\int_{\{x_{+}\}\times\om}\psi\,d\sigma=\int_{\{-x_{-}\}\times\om} \psi\,d\sigma=0\}$.
\begin{lemma}
We have the Poincar\'e inequality
\begin{equation}\label{PoincareMixt}
\int_{\Om_b} |\psi|^2\,dz \le \frac{1}{\mu_1}\,\int_{\Om_b} |\nabla\psi|^2\,dz,\qquad \forall \psi\in\mY,
\end{equation}
where $\mu_1>0$ is the smallest eigenvalue of the problem
\begin{equation}\label{EigenPb}
\begin{array}{|rcll}
\multicolumn{4}{|l}{\mbox{Find }(\mu,\zeta)\in \R\times (\mY\setminus\{0\}) \mbox{ such that } }\\[4pt]
-\Delta \zeta  & = & \mu \zeta & \mbox{ in }\Om_b\\[2pt]
\partial_{\nu} \zeta  & = & 0  & \mbox{ on }\partial \Om_b\cap\partial\Om\\[2pt]
\multicolumn{4}{|l}{\phantom{o}\partial_{\nu} \zeta\mbox{ is constant on }\{ x_{+}\}\times\om}\\[2pt]
\multicolumn{4}{|l}{\phantom{o}\partial_{\nu} \zeta\mbox{ is constant on }\{-x_{-}\}\times\om.}
\end{array}
\end{equation}
Here, $\nu$ stands for the normal unit vector to $\partial \Om_b$ directed to the exterior of $\Om_b$.
\end{lemma}
\noindent Applied to $u_s\in\mY$, estimate (\ref{PoincareMixt}) gives 
\begin{equation}\label{relation1bis}
\int_{\Om_b}|\nabla u_s|^2 - k^2|u_s|^2\,dz \ge (\mu_1- k^2)\int_{\Om_b} |u_s|^2\,dz.
\end{equation}
$\star$ Using (\ref{estim1}) and (\ref{relation1bis}) in (\ref{NullEnergy}), we obtain 
\begin{equation}\label{antepenul}
\begin{array}{lcl}
0  \ge 2(\lambda_1-k^2)\Big( \dsp\int_{\Om_+}|u_s|^2\,dz+\dsp\int_{\Om_-}|u_s|^2\,dz\Big) +(\mu_1-k^2)\dsp\int_{\Om_b} |u_\mrm{s}|^2\,dz.
\end{array}
\end{equation}
According to (\ref{antepenul}), for $k^2<\lambda_1$ such that $k^2\le\mu_1$, we must have $u_s\equiv0$. From the condition $\partial_{\nu} u=0$ on $\partial\Om\cap\partial\Om_b$ and the decomposition $u=u_s+u_i$, we infer $\partial_{\nu} u_i=\partial_{\nu} e^{ikx}=0$ on $\partial\Om\cap\partial\Om_b$. Thus, if we denote $\nu=(\nu_x,\nu_y)^{\top}\in\R\times\R^{d-1}$, we must have $\nu_xe^{ikx}=0$ on $\partial\Om\cap\partial\Om_b$. This is possible if and only if the defect is a family of cracks $\mathcal{F}$ parallel to the $(Ox)$ axis, which is excluded (see Remark \ref{RmkCrack}).
\end{proof}
\begin{remark}
Note that to derive the strongest result in the statement of Proposition \ref{MainProposition}, that is to obtain the largest constant $k_{\star}$, one has to choose $x_{\pm}$ right after (\ref{NullEnergy}) to get the largest first eigenvalue $\mu_1$ for Problem (\ref{EigenPb}).  However, since monotonicity results do not exist in general for Neumann eigenvalue problems (see e.g. \cite[\S1.3.2]{Henr06}), it is not clear how to set $x_{\pm}$. However, the above approach at least provides a lower bound for $k_{\star}$.
\end{remark}

\begin{remark}
Consider $(\Phi^{\eps})_{\eps\ge0}$ a sequence of smooth diffeomorphisms of $\R^d$ such that $\mrm{Id}-\Phi^{\eps}$ is compactly supported and such that $\|\mrm{Id}-\Phi^{\eps}\|_{\mrm{W}^{2,\infty}(\R^d)}\le \eps$ for small $\eps$. We also assume that the differential of $\Phi^{\eps}$ is uniformly bounded as $\eps$ goes to zero. Denote $\Om^{\eps}$ the image of the reference waveguide $\Om^0$ by $\Phi^{\eps}$. There is some $\ell>0$ such that $\Om^{\eps}$ coincides with $\Om^0$ for $z=(x,y)$ satisfying $|x|\ge \ell/2$. Let $\mu^{\eps}_1$ refer to the first eigenvalue of  Problem (\ref{EigenPb}) set in $\Om^{\eps}_{\ell/2}:=\{z\in\Om^{\eps}\,|\,|x|<\ell/2\}$. Then, according to the results of continuity of the Neumann eigenvalues of the Laplace operator with respect to smooth perturbations of the domain (see \cite[Theorem 2.3.25]{Henr06}), we have $|\mu^{\eps}_1-\mu^{0}_1| \le C\,\eps$, for some $C>0$, where $\mu^{0}_1=\min((\pi/\ell)^2,\lambda_1)$\footnote{Indeed, by a straightforward calculation, we find that the first eigenvalue of Problem (\ref{EigenPb}) set in $(-\ell/2;\ell/2)\times\om$ is equal to $\min((\pi/\ell)^2,\lambda_1)$.}. According to Proposition \ref{MainProposition}, we deduce that it is impossible to have $T=1$ for $k^2\in(0;\mu^{0}_1-C\,\eps)$. This shows that the technique considered in \cite{BoNa13}, based on the used of smooth and small perturbations of $\Om^0$ with a support of given width, cannot be efficient to construct waveguides such that $T=1$ for all wavenumbers $k\in(0;\lambda_1)$. 
\end{remark}

\section{Construction of perfectly invisible defects}\label{SectionConstruction}

\noindent In the previous section, we proved that for a given waveguide with sound hard walls (different from the reference waveguide with possible cracks $\mathcal{F}$ parallel to the $(Ox)$ axis), for wavenumbers smaller than a given bound $k_{\star}$, depending on the geometry, incident propagative waves always produce a scattered field which is not exponentially decaying  at $+\infty$. In particular, if the perturbation is smooth and small (in amplitude and in width), $k_{\star}$ is very close to the threshold wavenumber $\lambda_1$. In this section, we change the point of view. We choose any wavenumber between $0$ and the first threshold so that only one wave can propagate in the waveguide from $-\infty$ to $+\infty$. And we explain how to construct a waveguide with sound hard walls such that the scattered field associated with the incident field is exponentially decaying both at $-\infty$ and $+\infty$. To circumvent the obstruction established in the previous section, we will work with singular perturbations of the geometry. We will see that with the type of perturbations considered in this section, $k_{\star}$ can be made as small as we want (see Remark \ref{RmkObstruction} for details). We emphasize that here, the locution ``singular perturbation'' means that the perturbation of the geometry is such that a boundary layer phenomenon occurs. This does not mean that the geometry is not smooth (though we will consider a geometry with corners). For simplicity of exposition, we shall consider a basic 2D geometrical framework. Everything presented here can be extended to other settings.

\subsection{Setting of the problem}\label{SectionWaveguide}

\begin{figure}[!ht]
\centering
\begin{tikzpicture}[scale=2]
\draw[fill=gray!30,draw=none](-2,0.5) rectangle (2,2);
\draw (-2,0.5)--(2,0.5); 
\draw[dashed] (-3,0.5)--(-2,0.5); 
\draw[dashed] (3,0.5)--(2,0.5); 
\draw (-2,2)--(2,2);  
\draw[dashed] (3,2)--(2,2);
\draw[dashed] (-1.8,0.3)--(-1.8,2.2);
\draw[dashed] (1.8,0.3)--(1.8,2.2);
\draw[dashed] (-3,2)--(-2,2);

\draw[dotted,>-<] (-0.2,3.9)--(0.2,3.9);
\draw[dotted,>-<] (-0.8,3.4)--(-1.2,3.4);
\draw[dotted,>-<] (0.8,3.6)--(1.2,3.6);

\draw[dotted,>-<] (0.3,1.95)--(0.3,3.85);
\draw[dotted,>-<] (-0.7,1.95)--(-0.7,3.35);
\draw[dotted,>-<] (1.3,1.95)--(1.3,3.55);

\node at (0,4){\small $\eps$};
\node at (-1,3.5){\small $\eps$};
\node at (1,3.7){\small $\eps$};

\node at (-0.55,2.65){\small $h_1$};
\node at (0.45,2.9){\small $h_2$};
\node at (1.45,2.75){\small $h_3$};

\draw[fill=gray!30](-0.1,2) rectangle (0.1,3.8);
\draw[fill=gray!30](-0.9,2) rectangle (-1.1,3.3);
\draw[fill=gray!30](0.9,2) rectangle (1.1,3.5);

\draw[thick,draw=gray!30](-0.095,2)--(0.095,2);
\draw[thick,draw=gray!30](-0.905,2)--(-1.095,2);
\draw[thick,draw=gray!30](0.905,2)--(1.095,2);

\node[mark size=1pt,color=black] at (-1,2) {\pgfuseplotmark{*}};
\node[color=black] at (-1,1.8) {\small $M_1$};
\node[mark size=1pt,color=black] at (0,2) {\pgfuseplotmark{*}};
\node[color=black] at (0,1.8) {\small $M_2$};
\node[mark size=1pt,color=black] at (1,2) {\pgfuseplotmark{*}};
\node[color=black] at (1,1.8) {\small $M_3$};

\begin{scope}[shift={(2.5,1)}]
\draw[->] (0,0)--(0.5,0);
\draw[->] (0,0)--(0,0.5);
\node at (0.6,0){\small$x$};
\node at (0,0.6){\small$y$};
\end{scope}

\node at (-1.8,0){\small$x=-L/2$};
\node at (1.8,0){\small$x=L/2$};
\node at (-2,2.2){\small$\Gamma^{\eps}:=\partial\Omega^{\eps}$};
\end{tikzpicture}
\caption{Geometry of $\Om^{\eps}$.\label{DomainOriginal2D}} 
\end{figure}

\noindent Define the reference waveguide $\Om^{0}:=\{z=(x,y)\,|\,x\in\R\mbox{ and }y\in(0;1)\}\subset\R^2$. In this simple geometry, the first positive threshold is equal to $\pi^2$ and we will assume that the wavenumber $k$ satisfies 
\begin{equation}\label{AssumptionWaveNber2D}
0 < k < \pi.
\end{equation}
The modes are given by
\begin{equation}\label{defModes2D}
w^{\pm}_n(x,y)=\frac{e^{\pm i \beta_n x}}{ (2|\beta_n|)^{-1/2}}\,\varphi_n(y) \mbox{ with }\ \beta_n:=\sqrt{k^2-(n\pi)^2},\ \varphi_n(y):=\begin{array}{|ll}1 & \mbox{ if }n=0\\
\sqrt{2}\cos(n\pi y)  & \mbox{ if }n\ge 1\\
\end{array}.
\end{equation}
We remind the reader that we denote $w^{\pm}(x,y):=w^{\pm}_0(x,y)=(2k)^{-1/2}e^{\pm i k x}$ (notice that $|\om|^{-1/2}=1$ in the notation of Section \ref{SectionSetting}). Now, we define the perturbed waveguide 
\begin{equation}\label{defPerturbedDomain}
\Om^{\eps}=\Om^0\cup\mathscr{S}^{\eps}_1\cup \mathscr{S}^{\eps}_2\cup \mathscr{S}^{\eps}_3
\end{equation}
with, for $m=1,\dots,3$,
\begin{equation}\label{defThinStrips}
\mathscr{S}^{\eps}_m = (-\eps/2+x_m;x_m+\eps/2)\times[1;1+h_m).
\end{equation}
Here, the $x_m\in\R$ and $h_m>0$ are some numbers to determine (see Figure \ref{DomainOriginal2D} for an illustration). We assume that $x_1<x_2<x_3$ and we set $M_m=(x_m,1)$. In the following, we want to cancel the two complex coefficients $s^{\eps\pm}$. These coefficients are related by one law, namely the conservation of energy. Therefore, practically we have to cancel three real numbers. This is the reason why we add exactly three thin rectangles to $\Om^0$ (for more details, see the discussion after (\ref{FormulaCoeffs})). In the following,  we consider the scattering problem 
\begin{equation}\label{PbChampTotal2D}
\begin{array}{|rcll}
\multicolumn{4}{|l}{\mbox{Find }u^{\eps}\in\mrm{H}^1_{\loc}(\Om^{\eps}) \mbox{ such that }u^{\eps}-w^+\mbox{ is outgoing and } }\\[3pt]
-\Delta u^{\eps}  & = & k^2 u^{\eps} & \mbox{ in }\Om^{\eps}\\[3pt]
\partial_{\nu} u^{\eps}  & = & 0  & \mbox{ on }\Gamma^{\eps}:=\partial\Om^{\eps}.\end{array}
\end{equation}
For all $k\in(0;\pi)$, when 
\begin{equation}\label{NonResoning}
h_m\notin\{(2p+1)\pi/(2k)\,|\,p\in\N\}
\end{equation}
so that one avoids the resonances of the thin rectangles for problem (\ref{Pb1D}), working as in any of the papers  \cite{Beal73,Gady93,KoMM94,Naza96,Gady05,Naza05,BaNa15} (see \cite[Lemma 4.4]{JoTo06} for this particular result), we can prove that (\ref{PbChampTotal2D}) admits a unique solution $u^{\eps}$ for $\eps$ small enough. Moreover, we have the decomposition 
\begin{equation}\label{DecompoChampTotalBis}
u^{\eps}-\chi^-w^+ = \chi^+ T^{\eps} w^+ + \chi^- R^{\eps} w^- + \tilde{u}^{\eps},
\end{equation}
where $\tilde{u}^{\eps}\in\mH^1(\Om^{\eps})$ is a function which is exponentially decaying at $\pm\infty$. In the following, we use the notation  $u_i:=w^+$, $u^{\eps}_s:=u^{\eps}-w^+$. We also introduce the coefficients $s^{\eps\pm}$ such that
\begin{equation}\label{defCoeffspm}
s^{\eps-} = R^{\eps}\qquad \mbox{ and }\qquad s^{\eps+} = T^{\eps}-1.
\end{equation}
With this definition, according to (\ref{DecompoChampTotalBis}), we have $u^{\eps}_s= \chi^+ s^{\eps+} w^+ + \chi^- s^{\eps-} w^- + \hat{u}^{\eps}$ for some $\hat{u}^{\eps}\in\mH^1(\Om^{\eps})$ which is exponentially decaying at $\pm\infty$. In this particular domain, we can compute an asymptotic expansion of the coefficients $s^{\eps\pm}$ with respect to the small parameter $\eps$. The technique to derive such an expansion is quite classical but still requires to introduce a methodical exposition. The reader who wishes to skip the details may proceed directly to formula (\ref{expressionReflecCoeffterThree}). 

\subsection{Asymptotic expansion of $u^{\eps}$}\label{describExpansion}
We first construct an asymptotic expansion of $u^{\eps}$. The literature concerning asymptotic expansions for such problems is well-documented. For this reason, we give a rather compact presentation of the method. For more details, we refer the reader to the papers \cite{Beal73,Arse76,Gady93,KoMM94,Naza96,Gady05,Naza05,BaNa15}, especially to \cite{JoTo06,CDJT06,JoTo06b,JoTo08} where a problem very close to our concern, motivated by numerics, is considered. We search for an expansion of $u^{\eps}$ under the form
\begin{equation}\label{DefAnsatz}
u^{\eps}(z)= \begin{array}{|l}
u^{0}(z)+\eps\,u^{1}(z)+\dsp\sum_{m=1}^3\zeta_m(z)(V^0_m(\xi^m)+\eps\,(\ln\eps\,A_m+V^{1}_m(\xi^m)))+\dots,\  z\in\Om^0\\[12pt]
v^{0}_m(y)+\eps\,v^{1}_m(y)+\zeta_m(z)(V^0_m(\xi^m)+\eps\,(\ln\eps\,A_m+V^{1}_m(\xi^m)))+\dots,\  z\in\mathscr{S}^{\eps}_m,\ m=1,2,3,\hspace{-0.4cm}~
\end{array}
\end{equation}
where $z=(x,y)$ and where the dots correspond to small remainders. Note that the term $\ln\eps$ has been introduced in accordance with the above references. In (\ref{DefAnsatz}), for $m=1,2,3$, $\xi^m$ denotes the fast variable defined by $\xi^m=\eps^{-1}(z-M_m)$. The cut-off function $\zeta_m\in\mathscr{C}^{\infty}(\R^2,[0;1])$ is equal to one in a neighbourhood of $M_m$ and such that $\zeta_m(z)=0$ for $z$ satisfying $|z-M_m|\ge \min(1,|x_1-x_2|/2,|x_2-x_3|/2,h_m)$ or $|z|\ge L/2$. Now, we explain how to determine the functions $u^{0}$, $u^{1}$, $V^0_m$, $V^1_m$, $v^{0}_m$, $v^{1}_m$ as well as the constants $A_m$ appearing in (\ref{DefAnsatz}). We begin with a formal approach. The justification of the obtained results will be the purpose of Section \ref{SectionJustifAsym}. \\
\newline 
In first approximation, the thin rectangles are invisible for the incident field $w^+$. This yields $u^{\eps}=w^++\dots$ in $\Om^0$. Therefore, we impose $u^{0}=w^+=(2k)^{-1/2}e^{i k x}$. When $\eps$ goes to zero, $\mathscr{S}^{\eps}_m$ becomes the segment $\{x_m\}\times[1;1+h_m)$ and $u^{\eps}$ satisfies a one-dimensional Helmholtz equation. Matching the value of the fields at $M_m$ leads us to take $v^{0}_m$ such that
\begin{equation}\label{Pb1D}
\dsp\frac{d^2v^{0}_m}{dy^2}+k^2v^{0}_m=0\ \,\mbox{in }(1;1+h_m),\qquad \dsp\frac{dv^{0}_m}{dy}(1+h_m)=0\qquad\mbox{and}\qquad v^{0}_m(1)=w^{+}(M_m).
\end{equation}
This problem admits a unique solution $v^{0}_m(y)=w^{+}(M_m)(\cos(k(y-1))+\tan(kh_m)\sin(k(y-1)))$ if and only if $h_m$ satisfies condition (\ref{NonResoning}). In the following, we shall impose $h_m$ to meet this condition. Set 
\[
\Upsilon^{\eps}_m:=(-\eps/2+x_m;x_m+\eps/2)\times\{1\}
\]
A Taylor expansion of $u^{0}=w^+\in\mathscr{C}^{\infty}(\R^2)$ gives $u^{0}(x,1)=w^{+}(M_m)+(x-x_m)\partial_x w^{+}(M_m)+O(|x-x_m|^2)$. On $\Upsilon^{\eps}_m$, we have $|x-x_m|\le\eps$, which implies $|u^{0}-v^0_m|\le C\eps$. Thus, the jump of traces across $\Upsilon^{\eps}_m$ of the function equal to $u^0$ in $\Om^0$ and to $v^0_m$ in $\mathscr{S}^{\eps}_m$, is of order $\eps$. Therefore, we set $V^{0}_m=0$.  

\begin{figure}[!ht]
\centering
\begin{tikzpicture}[scale=1.5]
\begin{scope}
\clip (-2,2) -- (2,2) -- (2,0) -- (-2,0) -- cycle;
\draw[fill=gray!30,draw=none](0,2) circle (2);
\end{scope}
\draw[fill=gray!30,draw=none](-0.5,2) rectangle (0.5,3.5);
\draw (-2,2)--(-0.5,2)--(-0.5,3.5);
\draw (2,2)--(0.5,2)--(0.5,3.5);
\draw[dashed] (-0.5,3.5)--(-0.5,3.8);
\draw[dashed] (0.5,3.5)--(0.5,3.8);
\draw[dashed] (2.5,2)--(2,2);
\draw[dashed] (-2.5,2)--(-2,2);
\draw[thick,draw=gray!30](-0.495,2)--(0.495,2);
\node[mark size=1pt,color=black] at (0,2) {\pgfuseplotmark{*}};
\node[color=black] at (0,1.8) {\small $O$};
\begin{scope}[shift={(2.2,1)}]
\draw[->] (0,0)--(0.5,0);
\draw[->] (0,0)--(0,0.5);
\node at (0.66,0){\small$\xi_x$};
\node at (0,0.66){\small$\xi_y$};
\end{scope}
\node at (0,1){\small$\R^2_-$};
\node at (0,2.75){\small$\Pi$};
\node at (-2.5,2.9){\small$\Xi=\R^2_-\cup\Pi\cup(-1/2;1/2)\times\{0\}$};
\node at (2.5,2.9){\small\phantom{$\Xi=\R^2_-\cup\Pi\cup(-1/2;1/2)\times\{0\}$}};
\draw[dotted] (-0.5,2)--(0.5,2);
\end{tikzpicture}
\caption{Geometry of the frozen domain $\Xi$.\label{FrozenGeom}} 
\end{figure}
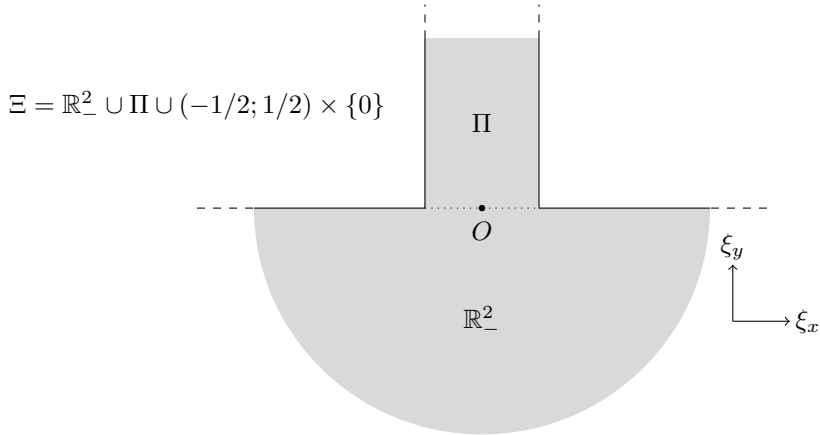

\noindent To compensate the jumps of traces and normal traces of the approximation of $u^{\eps}$ across $\Upsilon^{\eps}_m$ at order $\eps$, we will work with the function $V^{1}_m$. Let us make the change of variables $\xi_m=\eps^{-1}(z-M_m)$. When $\eps$ goes to zero, a neighbourhood of $M_m$ (where the cut-off function $\zeta_m$ is equal to one) is transformed into the domain $\Xi:=\R^2_-\cup\Pi\cup(-1/2;1/2)\times\{0\}$ which is the union of the half plane $\R^2_-:=\R\times(-\infty;0)$ and the half-strip $\Pi:=(-1/2;1/2)\times(0;+\infty)$ (see an illustration with Figure \ref{FrozenGeom}). Denote $\mrm{Id}$ the identity operator. In a neighbourhood of $M_m$, we have 
\[
\begin{array}{lcl}
(\Delta_z+k^2\mrm{Id})\,u^{\eps}(z) & =& (\Delta_z+k^2\mrm{Id})
(\eps\ln\eps\,A_m+\eps\,V^{1}_m(\eps^{-1}(z-M_m))+\dots)\\[3pt]
  &=& \eps^{-1}\,\Delta_{\xi_m}V^1_{m}(\xi_m)+\eps\ln\eps\,k^2A_m+\dots\ .
\end{array}
\]
Since there is no term of order $\eps^{-1}$ in the expansion (\ref{DefAnsatz}), we impose $\Delta_{\xi}V^1_{m}(\xi)=0$ for $\xi:=(\xi_x,\xi_y)\in\R^2_-\cup\Pi$. On $\Upsilon^{\eps}_m$, we have $u^0-v^0_m=(x-x_m)\partial_x w^{+}(M_m)+O(\eps^2)$ and $\partial_y(u^0-v^0_m)=-\partial_yv^0_m=- k\tan(kh_m)w^{+}(M_m)$. Remarking that $\partial_y=\eps^{-1}\partial_{\xi_y}$, we deduce that $V^1_{m}$ must satisfy the equations 
\begin{equation}\label{PbJump}
\begin{array}{|ll}
-\Delta V^{1}_m=0\quad\mbox{ in }\R^2_-\cup\Pi,\qquad\partial_{\nu}V^{1}_m=0\quad\mbox{ on  }\partial\Xi,\\[9pt]
V^{1}_m(\xi_x,0^-)-V^{1}_m(\xi_x,0^+)=-\xi_x\partial_x w^+(M_m),\quad\mbox{ for }\xi_x\in(-1/2;1/2)\\[8pt]
\partial_{\xi_y}V^{1}_m(\xi_x,0^-)-\partial_{\xi_y}V^{1}_m(\xi_x,0^+)=k\tan(kh_m)w^{+}(M_m),\quad\mbox{ for }\xi_x\in(-1/2;1/2).
\end{array}
\end{equation}
In Lemma \ref{lemmaTechnique} below, we prove that this problem admits a unique solution $V^{1}_m\in\mH^1_{\loc}(\R^2_-\cup\Pi)$ (the $\mH^1$-norm of $V^{1}_m$ is finite on each bounded subset of $\R^2_-\cup\Pi$) with the  behaviour
\begin{equation}\label{DecompoStripHPlane}
V^{1}_m(\xi) = \left\{\begin{array}{l}a_m\ln(|\xi|)+O(|\xi|^{-1}),\quad\xi\in\R^2_-,|\xi|\to+\infty,\\[4pt]
b_m+O(e^{-\pi\xi_y}),\quad\xi\in\Pi,\xi_y\to+\infty,
\end{array}\right.
\end{equation}
where $a_m$ and $b_m$ are some constants. In the following, the value of $a_m$ will be of particular interest. Integrating the equation $\Delta V^{1}_m(\xi)=0$  in the domain $\Xi_\ell=\{\xi\in\R^2_-\cup\Pi\,|\,|\xi| < \ell\}$ and taking the limit $\ell\to+\infty$, we obtain
\begin{equation}\label{defAm}
a_m=-\pi^{-1}k\tan(kh_m)w^{+}(M_m).
\end{equation}
Imposing $A_m=a_m$ in the expansion (\ref{DefAnsatz}) allows one to obtain an error on the source term in $\Om^0$ which is equal to $(\Delta_z+k^2\mrm{Id})(\sum_{m=1}^3\zeta_m(z)(\eps\ln\eps\,A_m+\eps\,V^{1}_m(\eps^{-1}(z-M_m))))=f^1(z)\,\eps+O(\eps^2)$ with 
\begin{equation}\label{PbChampTotalTerm1Bis}
f^1(z)=\dsp\sum_{m=1}^{3}a_m(\Delta_z+k^2\mrm{Id})\left(\zeta_m\ln r_m\right). 
\end{equation}
In (\ref{PbChampTotalTerm1Bis}), $r_m:=|z-M_m|$. Since $\zeta_m$ is equal to one in a neighbourhood of $M_m$ and compactly supported, the right hand side of (\ref{PbChampTotalTerm1Bis}) is an element of $\mrm{L}^2(\Om^0)$. On the other hand, choosing a cut-off function $\zeta_m$ which depends only on the variable $r_m=|z-M_m|$ in $\R^2_-$ implies $\partial_{\nu}(\zeta_m(z)V^1_m(\eps^{-1}(z-M_m)))=0$ on $\Gamma^{0}\setminus\overline{\Upsilon_m^{\eps}}$. Finally, to compensate the discrepancy (the error made on the source terms) at order $\eps$, we define $u^{1}$ as the solution to the problem
\begin{equation}\label{PbChampTotalTerm1}
\begin{array}{|rcll}
\multicolumn{4}{|l}{\mbox{Find }u^{1}\in\mrm{H}^1_{\loc}(\Om^{0}) \mbox{ such that }u^{1}\mbox{ is outgoing and } }\\[3pt]
-\Delta u^{1}-k^2 u^{1}  & = & f^1 & \mbox{ in }\Om^{0}\\[4pt]
\partial_{\nu} u^1  & = & 0  & \mbox{ on }\Gamma^{0}.
\end{array}
\end{equation}
In $\mathscr{S}^{\eps}_m$, without loss of generality, we can impose that the cut-off function $\zeta_m$ depends only on the variable $y$. Then we find an error on the source term equal to $(\Delta_z+k^2\mrm{Id})(\zeta_m(z)(\eps\ln\eps\,a_m+\eps\,V^{1}_m(\eps^{-1}(z-M_m))))=f^1_m(y)\,\eps\ln\eps+O(\eps)$ with 
\begin{equation}\label{PbChampTotalTerm1Ter}
f^1_m(y)=\dsp(\Delta_y+k^2\mrm{Id})(\zeta_m a_m). 
\end{equation}
A direct computation yields $\|f^1_m\,\eps\ln\eps\|_{\mrm{L}^2(\mathscr{S}^{\eps}_m)} \le C\,\eps^{3/2}(1+|\ln\eps|)$ where $C>0$ is independent of $\eps$. This error will be sufficiently small for our purpose and therefore we take $v^1_m\equiv0$ in (\ref{DefAnsatz}).\\
\newline
In the next section, from the asymptotic expansion of $u^{\eps}$, solution to (\ref{PbChampTotal2D}), we deduce an asymptotic expansion of the coefficients $s^{\eps\pm}$ appearing in the decomposition of $u^{\eps}$ (see (\ref{DecompoChampTotalBis})--(\ref{defCoeffspm})). Before proceeding, we explain how to prove an intermediate result which was used in the definition of the terms of (\ref{DecompoChampTotalBis})--(\ref{defCoeffspm}).

\begin{lemma}\label{lemmaTechnique}
Problem (\ref{PbJump}) has a unique solution $V^{1}_m\in\mH^1_{\loc}(\R^2_-\cup\Pi)$ admitting expansion (\ref{DecompoStripHPlane}).
\end{lemma}
\begin{proof}
We recall that $\Xi=\R^2_-\cup\Pi\cup(-1/2;1/2)\times\{0\}$. First, denote $\mathcal{N}\in\mH^1(\R^2_-\cup\Pi)$ the function such that $\mathcal{N}=0$ in $\R^2_-$ and 
\[
\Delta \mathcal{N}=0\ \mbox{ in }\ \Pi;\qquad \partial_{\xi_x}\mathcal{N}=0\ \mbox{ on }\ \partial\Pi\cap\partial\Xi;\qquad \mathcal{N}(\xi_x,0^+)=-\xi_x\partial_x w^+(M_m)\ \mbox{ for }\ \xi_x\in(-1/2;1/2).
\]
Since the integral of $\xi_x\partial_x w^+(M_m)$ on $(-1/2;1/2)$ is null, using decomposition in Fourier series, we can show that $\mathcal{N}$ is uniquely defined and exponentially decaying in $\Pi$ as $\xi_y\to+\infty$. Second, introduce $\phi\in\mathscr{C}^{\infty}(\Xi,[0;1])$, a cut-off function  such that $\phi=0$ in $\Pi\cup\mrm{D}_{1}(O)$ and $\phi=1$ in $\R^2_-\setminus\mrm{D}_{2}(O)$ ($\mrm{D}_{r}(O)$ denotes the open disk of $\R^2$ centered at $O$ of radius $r>0$). Additionally, we assume that $\phi$ depends only on the radial variable $|\xi|$. One sees that $V^{1}_m\in\mH^1_{\loc}(\R^2_-\cup\Pi)$ is a solution of (\ref{PbJump}) having behaviour (\ref{DecompoStripHPlane}) if and only if $W=V^{1}_m-\mathcal{N}+\pi^{-1}k\tan(kh_m)w^{+}(M_m)\,\phi\ln|\xi|\in\mH^1_{\loc}(\Xi)$ satisfies 
\begin{equation}\label{PbJumpBis}
\begin{array}{|ll}
-\Delta W=f\quad\mbox{ in }\R^2_-\cup\Pi,\qquad\partial_{\nu}W=0\quad\mbox{ on  }\partial\Xi,\\[9pt]
W(\xi_x,0^-)-W(\xi_x,0^+)=0,\quad\mbox{ for }\xi_x\in(-1/2;1/2)\\[8pt]
\partial_{\xi_y}W(\xi_x,0^-)-\partial_{\xi_y}W(\xi_x,0^+)=g,\quad\mbox{ for }\xi_x\in(-1/2;1/2).
\end{array}
\end{equation}
and admits the expansion
\begin{equation}\label{DecompoStripHPlaneBis}
W(\xi) = \left\{\begin{array}{l}A\ln(|\xi|)+O(|\xi|^{-1}),\quad\xi\in\R^2_-,|\xi|\to+\infty,\\[4pt]
B+O(e^{-\pi\xi_y}),\quad\xi\in\Pi,\xi_y\to+\infty,
\end{array}\right.
\end{equation}
where $A$, $B$ are some constants. In (\ref{PbJumpBis}), we set $f(\xi)=-\pi^{-1}k\tan(kh_m)w^{+}(M_m)\,\Delta(\phi\ln|\xi|)$ and $g(\xi)=k\tan(kh_m)w^{+}(M_m)+\partial_{\xi_y}\mathcal{N}(\xi_x,0^+)$. Now, one can check that $W\in\mH^1_{\loc}(\Xi)$ solves (\ref{PbJumpBis}) if and only if it satisfies the variational identity
\begin{equation}\label{DecompoStripHPlaneTer}
a(W,W')=\ell(W')\qquad\forall W'\in\mathscr{C}^{\infty}_0(\overline{\Xi}):=\{ w|_{\Xi}\,|\,w\in\mathscr{C}^{\infty}_0(\R^2)\}
\end{equation}
with $a(W,W')=\int_{\Xi}\nabla W\cdot\nabla W'\,d\xi$ and $\ell(W')=\int_{\Xi}f\,W'\,d\xi+\int_{(-1/2;1/2)\times\{0\}}g\,W'\,d\xi_x$. Define the space $\mZ$ as the completion of $\mathscr{C}^{\infty}_0(\overline{\Xi})$ in the weighted norm 
\[
\|w\|_{\mZ} :=\Big(\int_{\Xi}|\nabla w|^2+\cfrac{\phi|w|^2}{1+|\xi|^2\,(\ln|\xi|)^2}+\cfrac{(1-\phi)|w|^2}{1+|\xi_y|^2}\,d\xi\ \Big)^{1/2}.
\]
Note that $\mZ$ contains functions which are constant at infinity in $\R^2_-$ and in $\Pi$ but does not allow linear growth in $\Pi$ nor $\ln|\xi|$ behaviour in $\R^2_-$. Define the subspace $\mZ_{\#}:=\{w\in\mZ\,|\,\int_{K}w\,d\xi=0\}$ where $K\subset\overline{\Xi}$ is a compact set such that $\mrm{meas}_2(K)>0$. Using Hardy type inequalities, one can prove that the norms $\|\cdot\|_{\mZ}$ and $\|\nabla\cdot\|_{\mrm{L}^2(\Xi)}$ are equivalent in $\mZ_{\#}$. This allows us to show that there is a unique $\tilde{W}\in\mZ_{\#}$ such that $a(\tilde{W},W')=\ell(W')$ for all $W'\in\mZ_{\#}$. Observing that $\ell(1)=0$ (this is the reason why we introduced the shift $\pi^{-1}k\tan(kh_m)w^{+}(M_m)\,\phi\ln|\xi|$ in the definition of $W$), we conclude that there is a unique $W$ satisfying (\ref{DecompoStripHPlaneTer}) and admitting expansion (\ref{DecompoStripHPlaneBis}). This ends the proof. 
\end{proof}

\subsection{Asymptotic expansion of the reflection/transmission coefficients}

For the coefficients $s^{\eps\pm}$, working as in (\ref{limit1})-(\ref{limit2}), we find the formulas
\begin{equation}\label{expressionReflecCoeff}
 i\,s^{\eps\pm} = \dsp\int_{\Sigma_{-L}\cup\Sigma_{+L}} \frac{\partial u^{\eps}}{\partial\nu}\,\overline{w^{\pm}}-u^{\eps}\frac{\partial\overline{w^{\pm}}}{\partial\nu}\,d\sigma,
\end{equation}
where $\Sigma_{\pm L}=\{\pm L\}\times(0;1)$ and $\partial_{\nu}=\pm\partial_x$ at $x=\pm L$. Note that the right hand side of (\ref{expressionReflecCoeff}) remains unchanged with $L$ replaced by any $\ell>L/2$. Plugging the asymptotic expansion (\ref{DefAnsatz}) of $u^{\eps}$ in (\ref{expressionReflecCoeff}) and observing that the correction terms involving the cut-off functions $\zeta_m$ vanish for $|z|\ge L/2$, leads us to make the ansatz
\begin{equation}\label{expansionReflecCoeff}
s^{\eps\pm} = s^{0\pm} + \eps\,s^{1\pm}+\dots\ .
\end{equation}
And more precisely, identifying the powers in $\eps$ yields, for $j=0,1$, 
\[
 i\,s^{j\pm} = \dsp\int_{\Sigma_{-L}\cup\Sigma_{+L}} \frac{\partial u^{j}}{\partial\nu}\,\overline{w^{\pm}}- u^{j}\frac{\partial\overline{w^{\pm}}}{\partial\nu}\,d\sigma.
\]
Since $u^0=w^+$, first we deduce $s^{0\pm}=0$. In other words, we claim that $R^{\eps}=O(\eps)$ and $T^{\eps}-1=O(\eps)$. This seems reasonable since we make a perturbation of order $\eps$ in the reference waveguide where $R=0$ and $T=1$. On the other hand, integrating by parts in $\Om^0_L:=(-L;L)\times(0;1)$ and using (\ref{PbChampTotalTerm1Bis}), (\ref{PbChampTotalTerm1}), we obtain 
\begin{equation}\label{PredefQuantityI}
 i\,s^{1\pm} = -\dsp\int_{\Om^0_L}  f^1\,\overline{w^{\pm}}\,dz =
 -\dsp\sum_{m=1}^{3}a_m\,\mrm{I}_m
\end{equation}
with, for $m=1,2,3$,
\begin{equation}\label{defQuantityI}
\mrm{I}_m:=\dsp\int_{\Om^0_L}\overline{w^{\pm}}\,(\Delta_z+k^2\mrm{Id})(\zeta_m\ln r_m)\,dz= \dsp  \int_{\Om^0_L} \overline{w^{\pm}}\,\Delta(\zeta_m\ln r_m)-\zeta_m\ln r_m\,\Delta_z\overline{w^{\pm}}\,dz.
\end{equation}
Denote $\mrm{D}_{\delta}(M_m)$ the open disk of $\R^2$ centered at $M_m$ of radius $\delta>0$. Using the Lebesgue's dominated convergence theorem, we can write 
\begin{equation}\label{IntegrationByPartAutre}
\mrm{I}_m = \dsp\lim_{\delta\to 0}\ \dsp  \int_{\Om^0_L\setminus\overline{\mrm{D}_{\delta}(M_m)}} \overline{w^{\pm}}\,\Delta(\zeta_m\ln r_m)-\zeta_m\ln r_m\,\Delta\overline{w^{\pm}}\,dz.
\end{equation}
Integrating by parts in (\ref{IntegrationByPartAutre}) and using that $\zeta_m$ is equal to one in a neighbourhood of $M_m$ as well as the fact that $\zeta_m$ vanishes on $\Sigma_{\pm L}$, we get 
\begin{equation}\label{explicitCalculus}
\mrm{I}_m   =  \dsp\lim_{\delta\to 0} \  \dsp\int_{\partial\mrm{D}_{\delta}(M_m)\cap\Om^0_{L}} -\overline{w^{\pm}}\,\partial_{r_m}(\ln r_m)+\ln r_m\,\partial_{r_m}\overline{w^{\pm}}\,d\sigma.
\end{equation}
Then, a direct computation gives $\mrm{I}_m=-\pi\,\overline{w^{\pm}}(M_m)$. Plugging this result in (\ref{PredefQuantityI}), (\ref{defQuantityI}), we find 
\[
 i\,s^{1\pm} =  -k\sum_{m=1}^{3}\overline{w^{\pm}}(M_m)w^+(M_m)\tan(kh_m).
\]
Summing up, when $\eps$ goes to zero, the coefficients $s^{\eps\pm}$ appearing in the decomposition of $u^{\eps}$ (see (\ref{DecompoChampTotalBis})--(\ref{defCoeffspm})) admit the asymptotic expansion 
\begin{equation}\label{expressionReflecCoeffterThree}
i\,s^{\eps\pm} = -\eps k\sum_{m=1}^{3}\overline{w^{\pm}}(M_m)w^+(M_m)\tan(kh_m)+\dots\ .
\end{equation}

\subsection{The fixed point procedure}

Observing (\ref{expressionReflecCoeffterThree}), we see it is easy to find $h_1,h_2,h_3$ such that the coefficients $s^{\eps\pm}$ vanish at order $\eps$. For example, one can take $h_1=h_2=h_3=\pi/k$ (note that $h_m=p_m\pi/k$ for $p_m\in\N$ with some $p_m\ne0$ is also a valid choice). However, this is not sufficient since we want to impose $s^{\eps\pm}=0$ (at any order in $\eps$). To control the higher order terms in $\eps$ whose dependence with respect to $h_1,h_2,h_3$ is less simple than for the first term of the asymptotics, we will use the fixed point theorem. To obtain a fixed point formulation, for $m=1,2,3$, we look for $h_m$ under the form 
\begin{equation}\label{ExpressEll}
h_m = \frac{\pi}{k} + \tau_m.
\end{equation}
In these expressions, $\tau_{1},\tau_{2},\tau_{3}$ are real parameters that we will tune to impose $s^{\eps\pm}=0$. We define the vector $\tau=(\tau_1,\tau_2,\tau_3)^{\top}\in\R^3$ and we denote $\Om^{\eps}(\tau)$, $u^{\eps}(\tau)$, $s^{\eps\pm}(\tau)$ instead of $\Om^{\eps}$, $u^{\eps}$, $s^{\eps\pm}$. With this particular choice for $h_1,h_2,h_3$, plugging (\ref{ExpressEll}) in (\ref{expressionReflecCoeffterThree}), we obtain 
\begin{equation}\label{FormulaCoeffs}
is^{\eps\pm}(\tau) =  -\eps k\sum_{m=1}^{3}\overline{w^{\pm}}(M_m)w^+(M_m)\tan(k\tau_m)+\tilde{s}^{\eps\pm}(\tau),
\end{equation}
where $\tilde{s}^{\eps\pm}(\tau)$ is a remainder. We want to impose $s^{\eps\pm}(\tau) =0$. According to the energy conservation, we have $|s^{\eps-}(\tau)|^2+|1+s^{\eps+}(\tau)|^2=1\Leftrightarrow |s^{\eps-}(\tau)|^2+|s^{\eps+}(\tau)|^2=-2\Re e\,s^{\eps+}(\tau)$. Therefore, if we impose $\Re e\,s^{\eps-}(\tau)=\Im m\,s^{\eps-}(\tau)=\Im m\,s^{\eps+}(\tau)=0$, then there holds $s^{\eps-}(\tau)=0$ and $s^{\eps+}(\tau)\in\{0,-2\}$. In order to discard the case $s^{\eps+}(\tau)=-2$, we will use the fact that the modulus of $s^{\eps+}(\tau)$ is small for $\eps$ small enough\footnote{Of course, it seems much simpler to impose directly $\Re e\,s^{\eps+}(\tau)=0$. Indeed, with the energy conservation, this implies $s^{\eps\pm}(\tau)=0$. However, our approach does not allow us to do it because the first term in the asymptotic expansion of $s^{\eps+}(\tau)$ is purely imaginary (see formula (\ref{FormulaCoeffs})).}. Since $w^{\pm}(M_m)=(2k)^{-1/2} e^{\pm i k x_m}$, we see from (\ref{FormulaCoeffs}) that we have to find $\tau=(\tau_1,\tau_2,\tau_3)^{\top}\in\R^3$ such that 
\begin{equation}\label{pbFixedPoint}
\mathbb{M}\ (\tan(k\tau_1),\tan(k\tau_2),\tan(k\tau_3))^{\top} = 2\eps^{-1}\,(\Re e\,\tilde{s}^{\eps-}(\tau),\Im m\,\tilde{s}^{\eps-}(\tau),\Re e\,\tilde{s}^{\eps+}(\tau))^{\top}
\end{equation}
where we denote 
\begin{equation}\label{defMatrixM}
\mathbb{M}:=\left(\begin{array}{ccc}
\cos(2kx_1) & \cos(2kx_2) & \cos(2kx_3) \\[7pt]
\sin(2kx_1) & \sin(2kx_2) & \sin(2kx_3) \\[7pt]
1 & 1 & 1 
\end{array}\right).
\end{equation}
To proceed, first we set $x_1$, $x_2$, $x_3$, the numbers which determine the positions of the thin rectangles (see (\ref{defThinStrips})), so that the matrix $\mathbb{M}$ is invertible. This can be easily done, taking for example $x_1=-(2p+1)\pi/(4k)$, $x_2=0$ and $x_3=-x_1$ with $p\in\N$. Indeed, in this case we get the invertible matrix 
\begin{equation}\label{ChoiceMatrix}
\mathbb{M}=\left(\begin{array}{ccc}
0 & 1 & 0 \\[7pt]
(-1)^{p+1}& 0 & (-1)^{p} \\[7pt]
1 & 1 & 1 
\end{array}\right).
\end{equation}
For $m=1,2,3$ and $\tau_m\in(-\pi/(2k);\pi/(2k))$, define 
\begin{equation}\label{defttau}
\mathfrak{t}_m=\tan(k\tau_m).
\end{equation}
Denote $\mathfrak{t}=(\mathfrak{t}_1,\mathfrak{t}_2,\mathfrak{t}_3)^{\top}$. From (\ref{pbFixedPoint}), we see that $\mathfrak{t}$ must be a solution to the problem \\
\begin{equation}\label{PbFixedPoint}
\begin{array}{|l}
\mbox{Find }\mathfrak{t}\in\R^{3}\mbox{ such that }\mathfrak{t} = \mathscr{F}^{\eps}(\mathfrak{t}),
\end{array}~\\[5pt]
\end{equation}
with 
\begin{equation}\label{DefMapContra}
\mathscr{F}^{\eps}(\mathfrak{t}):= 2\eps^{-1}\,\mathbb{M}^{-1}(\Re e\,\tilde{s}^{\eps-}(\tau),\Im m\,\tilde{s}^{\eps-}(\tau),\Re e\,\tilde{s}^{\eps+}(\tau))^{\top}. 
\end{equation}
Proposition \ref{propoTech} hereafter ensures that there is some $\gamma_0>0$ such that for all $\gamma\in(0;\gamma_0]$, the map $\mathscr{F}^{\eps}$ is a contraction of $\mrm{B}_\gamma:=\{\mathfrak{t}\in\R^{3}\,\big|\,|\mathfrak{t}| \le \gamma\}$ for $\eps$ small enough. Therefore, the Banach fixed-point theorem guarantees the existence of some $\eps_{\gamma}>0$ such that for all $\eps\in(0;\eps_{\gamma}]$, Problem (\ref{PbFixedPoint}) has a unique solution $\mathfrak{t}^{\mrm{sol}}$ in $\mrm{B}_\gamma$. From this vector $\mathfrak{t}^{\mrm{sol}}=(\mathfrak{t}^{\mrm{sol}}_1,\mathfrak{t}^{\mrm{sol}}_2,\mathfrak{t}^{\mrm{sol}}_3)^{\top}$, define $(\tau^{\mrm{sol}}_1,\tau^{\mrm{sol}}_2,\tau^{\mrm{sol}}_3)\in(-\pi/(2k);\pi/(2k))^3$ such that $\mathfrak{t}^{\mrm{sol}}_m=\tan(k\tau^{\mrm{sol}}_m)$. Since for $\eps$ small enough, we have $|s^{\eps+}(\tau^{\mrm{sol}})|=O(\eps)$, we infer that the coefficients $s^{\eps\pm}(\tau^{\mrm{sol}})$ satisfy $s^{\eps\pm}(\tau^{\mrm{sol}})=0$.\\
\newline
Observe that the height $h^{\mrm{sol}}_m := \pi/k + \tau^{\mrm{sol}}_m$ of the rectangle $\mathscr{S}^{\eps}_m(\tau^{\mrm{sol}})$ is different from zero because $\tau^{\mrm{sol}}\in(-\pi/(2k);\pi/(2k))^3$ (more precisely, in Remark \ref{remarkFirstOrder} below, we show that we have $|\tau^{\mrm{sol}}|=O(\eps^{1/2}(1+|\ln\eps|)$). As a consequence, we have constructed a waveguide, which is not the reference waveguide, where the reflection/transmission coefficients in the decomposition (\ref{DecompoChampTotalBis}) of $u^{\eps}(\tau^{\mrm{sol}})$ satisfy $R^{\eps}(\tau^{\mrm{sol}})=0$ and $T^{\eps}(\tau^{\mrm{sol}})=1$. We summarize this result in the following proposition.  
\begin{proposition}\label{propositionMainResult}
For any $k\in(0;\pi)$, there exists a waveguide $\Om^{\eps}(\tau^{\mrm{sol}})$, different from $\Om^0$, where the scattered field associated with Problem (\ref{PbChampTotal2D}) set in $\Om^{\eps}(\tau^{\mrm{sol}})$ is exponentially decaying at $\pm\infty$. 
\end{proposition}

\begin{remark}\label{RmkObstruction}
Let us check that the result of Proposition \ref{propositionMainResult} is not incompatible with the one of Proposition \ref{MainProposition} (obstruction to transmission invisibility). Define the function $\psi$ such that $\psi(z)=0$ in $\Om^0$ and $\psi(z)=\sin(\pi(y-1)/(2h^{\mrm{sol}}_m))$ in $\mathscr{S}^{\eps}_m$. Note that $\psi$ belongs to the space $\mY=\{\psi\in\mH^1(\Om_b)\,|\,\int_{\{x_{+}\}\times\om}\psi\,d\sigma=\int_{\{-x_{-}\}\times\om} \psi\,d\sigma=0\}$ appearing just before (\ref{EigenPb}), where here $\Om_b=\{z\in\Om^{\eps}(\tau^{\mrm{sol}}\})\,|\,-x_-<x<x_+ \}$. A direct calculation yields
\[
\int_{\Om_b}|\nabla \psi|^2\,dx \le \underset{m=1,2,3}{\max}(\pi/(2h^{\mrm{sol}}_m))^2\int_{\Om_b}|\psi|^2\,dx .
\]
Using the \textit{min-max} principle, we deduce that the first eigenvalue $\mu_1$ of Problem (\ref{EigenPb}) is smaller than $\max_{m=1,2,3}(\pi/(2h^{\mrm{sol}}_m))^2=k^2/4+O(\eps^{1/2}(1+|\ln\eps|))$. From Proposition \ref{MainProposition}, we infer that in $\Om^{\eps}(\tau^{\mrm{sol}})$ (this geometry is defined for a given/frozen $k$) we cannot have transmission invisibility for wavenumbers $\tilde{k}$ such that $\tilde{k}\lesssim k/2$. Of course, this does not prevent from having transmission invisibility for $\tilde{k}=k$. 
\end{remark}

\begin{remark}
Let us recall that the positions $x_1$, $x_2$, $x_3$ of the chimneys must be chosen in order to ensure the invertibility of the matrix $\mathbb{M}$ given by (\ref{defMatrixM}). If we impose $x_3-x_2=x_2-x_1=\eta>0$, we obtain $\det\,\mathbb{M}=2 \sin(2k \eta)(1-\cos(2k \eta))$. 
As a consequence, the matrix $\mathbb{M}$ is invertible as soon as $\eta\notin \{p\pi/2k\,|\,p\in \N\}$: the distance of two consecutive chimneys should not be a multiple of a quarter of the wavelength. 
\end{remark}

\section{Numerical experiments}\label{SectionNumericalExperiments}
We implement numerically the approach developed in Section 
\ref{SectionConstruction}. For a given wavenumber $k\in(0;\pi)$, we consider the scattering problem 
\begin{equation}\label{PbChampTotalNumerical}
\begin{array}{|rcll}
\multicolumn{4}{|l}{\mbox{Find }u^{\eps}\in\mrm{H}^1_{\loc}(\Om^{\eps}) \mbox{ such that }u^{\eps}-w^+\mbox{ is outgoing and } }\\[3pt]
-\Delta u^{\eps}  & = & k^2 u^{\eps} & \mbox{ in }\Om^{\eps}\\[3pt]
\partial_{\nu} u^{\eps}  & = & 0  & \mbox{ on }\Gamma^{\eps}.\end{array}
\end{equation}
Our goal is to build $\Om^{\eps}$ such that the scattered field $u^{\eps}_{\mrm{s}}=u^{\eps}-w^+$ is exponentially decaying at $\pm\infty$.  Following (\ref{defPerturbedDomain}), we search for $\Om^{\eps}$ of the form $\Om^{\eps}(\tau)=\Om^0\cup \mathscr{S}^{\eps}_1(\tau)\cup \mathscr{S}^{\eps}_2(\tau)\cup \mathscr{S}^{\eps}_3(\tau)$, with $\Om^0=\R\times(0;1)$ and, for $m=1,\dots,3$,
\begin{equation}\label{expressionDomain}
\mathscr{S}^{\eps}_m(\tau) = (-\eps/2+x_m;x_m+\eps/2)\times[1;1+\pi/k+\tau_m).
\end{equation}
In accordance with what precedes (\ref{ChoiceMatrix}), we choose $x_1=-3\pi/(4k)$, $x_2=0$ and $x_3=-x_1$. 
Now, to determine the parameters $\tau_m$ of (\ref{expressionDomain}), we solve the fixed point problem (\ref{PbFixedPoint}) using a recursive procedure.\\
\newline
We denote $\tau^{j}=(\tau^{j}_{1},\tau^{j}_{2},\tau^{j}_{3})^{\top}\in(-\pi/(2k);\pi/(2k))^3$ the value of $\tau=(\tau_{1},\tau_{2},\tau_{3})^{\top}$ at iteration $j\ge0$. For $m=1,2,3$, we define $\mathfrak{t}^{j}_m=\tan(k\tau^{j}_m)$ and we denote $\mathfrak{t}^{j}=(\mathfrak{t}^{j}_1,\mathfrak{t}^{j}_2,\mathfrak{t}^{j}_3)^{\top}$. We set $\mathfrak{t}^{0}=(0,0,0)^{\top}$. Then, in accordance with (\ref{PbFixedPoint}), recursively we define $\mathfrak{t}^{j}$ using the formula $\mathfrak{t}^{j+1} = \mathscr{F}^{\eps}(\mathfrak{t}^j)$. From (\ref{FormulaCoeffs}), (\ref{DefMapContra}), one obtains that this is equivalent to define 
\begin{equation}\label{EcriturePtFixe}
\begin{array}{|lcl}
\mathfrak{t}^{j+1}=\mathfrak{t}^{j}+2\eps^{-1}\mathbb{M}^{-1}(\Re e\,(is^{\eps-}(\tau^{j})),\Im m\,(is^{\eps-}(\tau^j)),\Re e\,(is^{\eps+}(\tau^{j})))^{\top}.
\end{array}
\end{equation}
In (\ref{EcriturePtFixe}), the coefficients $s^{\eps\pm}(\tau^j)$ are computed at each step $j\ge0$ solving the scattering problem 
\begin{equation}\label{PbChampTotalNumericalSteps}
\begin{array}{|rcll}
\multicolumn{4}{|l}{\mbox{Find }u^{\eps}(\tau^j)\in\mrm{H}^1_{\loc}(\Om^{\eps}(\tau^j)) \mbox{ such that }u^{\eps}(\tau^j)-w^+\mbox{ is outgoing and } }\\[3pt]
-\Delta u^{\eps}(\tau^j)  & = & k^2 u^{\eps}(\tau^j) & \mbox{ in }\Om^{\eps}(\tau^j)\\[3pt]
\partial_{\nu} u^{\eps}(\tau^j)  & = & 0  & \mbox{ on }\Gamma^{\eps}(\tau^j):=\partial\Om^{\eps}(\tau^j).\end{array}
\end{equation}
Note that the domain $\Om^{\eps}(\tau^j)$ depends on the step $j\ge0$. More precisely, $\Om^{\eps}(\tau^j)$ is defined by $\Om^{\eps}(\tau^j)=\Om^0\cup \mathscr{S}^{\eps}_1(\tau^j)\cup \mathscr{S}^{\eps}_2(\tau^j)\cup \mathscr{S}^{\eps}_3(\tau^j)$, with, for $m=1,\dots,3$,
\[
\mathscr{S}^{\eps}_m(\tau^j) =(-\eps/2+x_m;x_m+\eps/2)\times[1;1+\pi/k+\tau^j_m)\mbox{ and }\tau^j_m=\arctan(\mathfrak{t}^{j}_m)/k\in(-\pi/(2k);\pi/(2k)).
\]
Then according to formula (\ref{expressionReflecCoeff}), the coefficients $s^{\eps\pm}(\tau^j)$ are given by
\begin{equation}\label{DefTermeChampLointain}
\begin{array}{lcl}
 i\,s^{\eps\pm}(\tau^j) & = &\dsp\int_{\Sigma_{-L}\cup\Sigma_{+L}} \frac{\partial u^{\eps}_s(\tau^j)}{\partial\nu}\,\overline{w^{\pm}}-u_s^{\eps}(\tau^j)\frac{\partial\overline{w^{\pm}}}{\partial\nu}\,d\sigma\\[12pt]
 &= &\dsp\int_{\Sigma_{-L}\cup\Sigma_{+L}} \frac{\partial u^{\eps}(\tau^j)}{\partial\nu}\,\overline{w^{\pm}}-u^{\eps}(\tau^j)\frac{\partial\overline{w^{\pm}}}{\partial\nu}\,d\sigma.
\end{array}
\end{equation}
At each step $j\ge0$, we approximate the solution of Problem (\ref{PbChampTotalNumericalSteps}) with a P2 finite element method in $\Om^{\eps}_5(\tau^j):=\{z=(x,y)\in\Om^{\eps}(\tau^j)\,|\,|x|<5\}$. At $x=\pm 5$, a truncated Dirichlet-to-Neumann map with 20 terms serves as a transparent boundary condition. We emphasize that we consider such a long domain (numerically, this is no necessity to do this) just to obtain nice pictures. For the numerical experiments, the wavenumber $k$ is set to $k=0.8\,\pi$. For the simulations of Figures \ref{figResult1}--\ref{figResult2}, we take $\eps=0.3$ and we stop the procedure when $\sum_{m=1}^3|\mathfrak{t}^{j+1}_{m}-\mathfrak{t}^{j}_{m}|\le 10^{-9}$ (corresponding here to 12 iterations). To obtain the results of Figure \ref{figResult3}, we perform 15 iterations and we try several values of $\eps$. Computations are implemented with \textit{FreeFem++}\footnote{\textit{FreeFem++}, \url{http://www.freefem.org/ff++/}.} while  results are displayed with \textit{Matlab}\footnote{\textit{Matlab}, \url{http://www.mathworks.com/}.} and \textit{Paraview}\footnote{\textit{Paraview}, \url{http://www.paraview.org/}.}.

\begin{figure}[!ht]
\centering
\includegraphics[width=15cm]{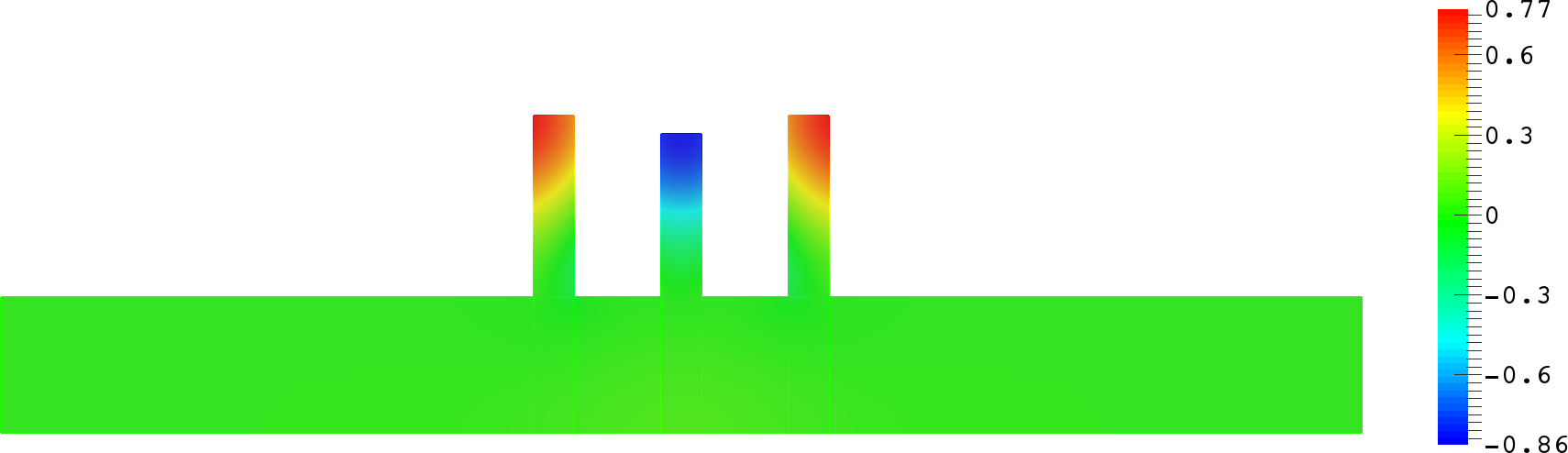}\\[14pt]\includegraphics[width=15cm]{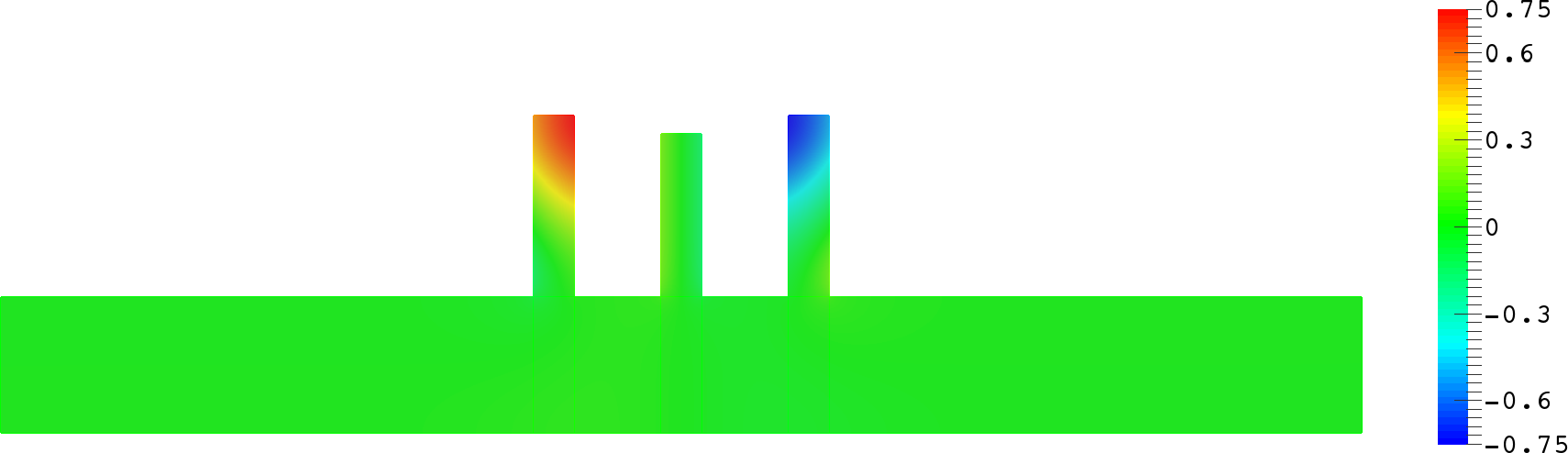} 
\caption{Real part (top) and imaginary part (bottom) of the approximation of the scattered field $u^{\eps}_s=u^{\eps}-u_i$ at the end of the fixed point procedure (12 iterations). As expected, the amplitude of the field is very small at $x=\pm5$. Interestingly, the fixed point procedure converges though the parameter $\eps$ (the width of the vertical rectangles) is not very small (here $\eps=0.3$). 
\label{figResult1}}
\end{figure}

\begin{figure}[!ht]
\centering
\includegraphics[width=15cm]{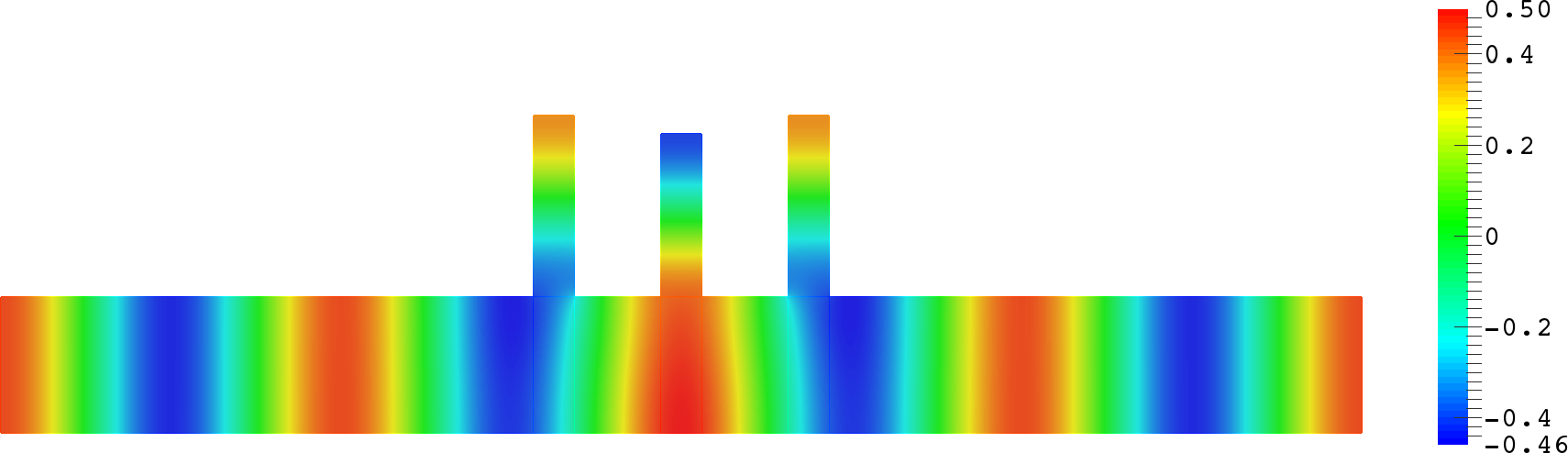}\\[14pt]\includegraphics[width=15cm]{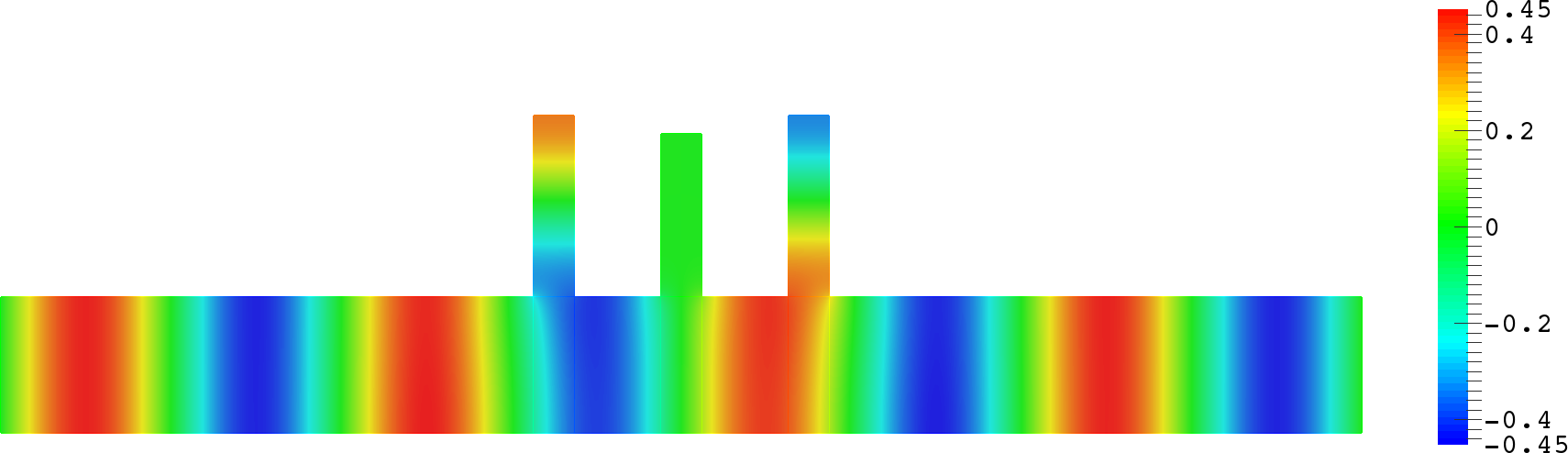}
\caption{Real part (top) and imaginary part (bottom) of the approximation of the total field $u^{\eps}$ at the end of the fixed point procedure (12 iterations). 
\label{figResult2}
}
\end{figure}

\begin{figure}[!ht]
\centering
\includegraphics[width=15cm]{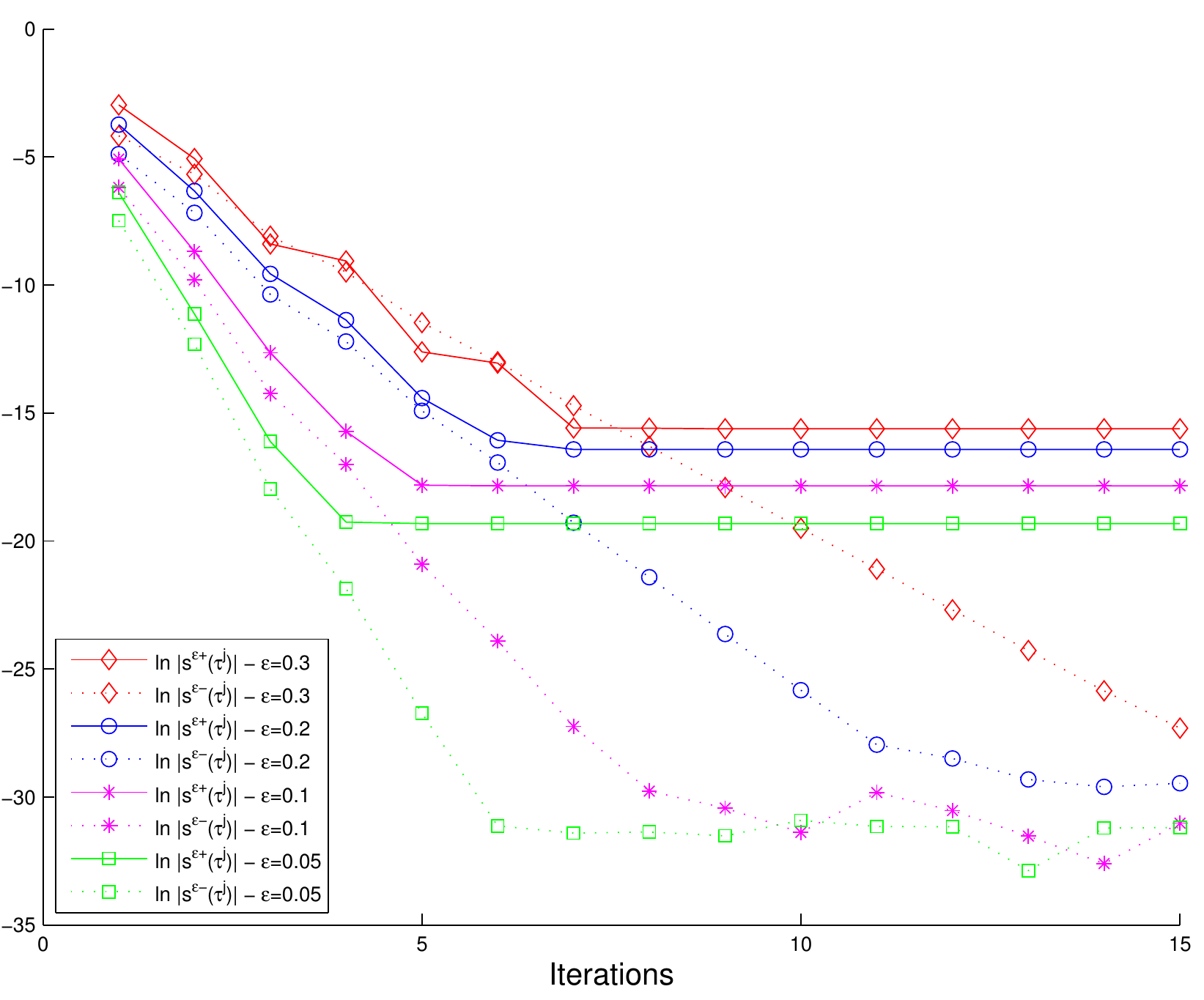}
\caption{Curves $\ln |s^{\eps+}(\tau^j)|$ and $\ln |s^{\eps-}(\tau^j)|$ with respect to the number of iterations $j$ for several values of $\eps$. The smaller $\eps$, the faster the iterative procedure converges.\label{figResult3}}
\end{figure}

\newpage~

\section{Justification of asymptotics}\label{SectionJustifAsym}

\noindent In this section, we show that the map $\mathscr{F}^{\eps}$ is a contraction as required in the analysis leading to Proposition \ref{propositionMainResult}.
\begin{proposition}\label{propoTech}
Consider some $\gamma>0$ sufficiently small. Then, there exists $\eps_{\gamma}>0$ such that for all $\eps\in(0;\eps_{\gamma}]$, the map $\mathscr{F}^{\eps}$ defined by (\ref{DefMapContra}) is a contraction of $\mrm{B}_\gamma:=\{\mathfrak{t}\in\R^{3}\,\big|\,|\mathfrak{t}| \le \gamma\}$.
\end{proposition}
\begin{proof}
For $\mathfrak{t}=(\mathfrak{t}_1,\mathfrak{t}_2,\mathfrak{t}_3)^{\top}\in\R^{3}$, we have $\mathscr{F}^{\eps}(\mathfrak{t})= 2\eps^{-1}\,\mathbb{M}^{-1}(\Re e\,\tilde{s}^{\eps-}(\tau),\Im m\,\tilde{s}^{\eps-}(\tau),\Re e\,\tilde{s}^{\eps+}(\tau))^{\top}$, where $\mathbb{M}$ is defined in (\ref{defMatrixM}), $\tau=(\tau_1,\tau_2,\tau_3)^{\top}\in(-\pi/(2k);\pi/(2k))^3$ is such that $\mathfrak{t}_m=\tan(k\tau_m)$ for $m=1,2,3$ and $\tilde{s}^{\eps\pm}(\tau)$ are the remainders appearing in (\ref{FormulaCoeffs}). In the following (see \S\ref{paragraphDefiRemainder}), we will prove that for $\vartheta$ small enough, there is some $\eps_{\vartheta}>0$ such that for all $\eps\in(0;\eps_{\vartheta}]$, there holds 
\begin{equation}\label{MainEstimate}
|\tilde{s}^{\eps\pm}(\tau)-\tilde{s}^{\eps\pm}(\tau')| \le C\,\eps^{3/2}\,(1+|\ln\eps|)\,|\tau-\tau'|,\qquad \forall\tau,\,\tau'\in \mrm{B}_{\vartheta}=\{\tau\in\R^{3}\,\big|\,|\tau| \le \vartheta\}.
\end{equation}
Here and in what follows, $C>0$ is a constant which may change from one occurrence to another but which is independent of $\eps$ and $\mathfrak{t},\mathfrak{t}'\in\mrm{B}_\gamma$. Since $\tau$ belongs to $\mrm{B}_{k^{-1}\arctan(\gamma)}$ when $\mathfrak{t}\in \mrm{B}_\gamma$, (\ref{MainEstimate}) yields 
\begin{equation}\label{EstimContra}
|\mathscr{F}^{\eps}(\mathfrak{t})-\mathscr{F}^{\eps}(\mathfrak{t}')|\le C\,\eps^{1/2}\,(1+|\ln\eps|)\,|\mathfrak{t}-\mathfrak{t}'|,\qquad \forall\mathfrak{t},\mathfrak{t}'\in\mrm{B}_\gamma,
\end{equation}
when $\gamma$ is chosen sufficiently small. Taking $\mathfrak{t}'=0$ in (\ref{EstimContra}) and remarking that $|\mathscr{F}^{\eps}(0)| \le C\,\eps^{1/2}\,(1+|\ln\eps|)$ (use Proposition \ref{PropositionErrorEstimate} below to show this), we find $|\mathscr{F}^{\eps}(\mathfrak{t})| \le C\,\eps^{1/2}\,(1+|\ln\eps|)$ for all $\mathfrak{t}\in\mrm{B}_\gamma$. With (\ref{EstimContra}), this concludes the proposition.
\end{proof}
\begin{remark}\label{remarkFirstOrder}
\noindent Let us denote $\mathfrak{t}^{\mrm{sol}}\in\mrm{B}_\gamma$ the unique solution to Problem (\ref{PbFixedPoint}). The previous proof ensures that we have
\begin{equation}\label{RelAPos}
|\mathfrak{t}^{\mrm{sol}}| = |\mathscr{F}^{\eps}(\mathfrak{t}^{\mrm{sol}})| \le C\,\eps^{1/2}\,(1+|\ln\eps|),\qquad \forall\eps\in(0;\eps_{\gamma}].
\end{equation}
Introduce $\tau^{\mrm{sol}}\in(-\pi/(2k);\pi/(2k))^3$ the vector such that $\mathfrak{t}^{\mrm{sol}}_m=\tan(k\tau^{\mrm{sol}}_m)$ for $m=1,2,3$. From (\ref{RelAPos}), we obtain $|\tau^{\mrm{sol}}|\le C\,\eps^{1/2}\,(1+|\ln\eps|)$. As a consequence, we can say that the height of the small rectangles $h^{\mrm{sol}}_m := \pi/k + \tau^{\mrm{sol}}_m$ satisfies $h^{\mrm{sol}}_m \approx \pi/k$ as $\eps$ tends to zero. 
\end{remark}
\noindent It remains to establish estimate (\ref{MainEstimate}), the main task of the section. The proof will be divided into several steps and will be the concern of the next three parts. 

\subsection{Reduction to a bounded domain}

We set $\Om^{\eps}_{L}(\tau):=\{z=(x,y)\in\Om^{\eps}(\tau)\,|\,|x|\le L\}$ where $\Om^{\eps}(\tau)$ is defined just before (\ref{expressionDomain}). Classically, one shows that $u^{\eps}(\tau)$ satisfies the scattering problem (\ref{PbChampTotal2D}) in $\Om^{\eps}(\tau)$ if and only if it is a solution to 
\begin{equation}\label{FormulationRef}
\begin{array}{|rcll}
\multicolumn{4}{|l}{\mbox{Find }u^{\eps}(\tau)\in\mrm{H}^1(\Om^{\eps}_L(\tau)) \mbox{ such that } }\\[3pt]
-\Delta u^{\eps}(\tau)  & = & k^2 u^{\eps}(\tau) & \mbox{ in }\Om^{\eps}_L(\tau)\\[3pt]
\partial_{\nu} u^{\eps}(\tau)  & = & 0  & \mbox{ on }\partial\Om^{\eps}_L(\tau)\cap\partial\Om^{\eps}(\tau)\\[3pt]
\partial_{\nu}(u^{\eps}(\tau)-w^+)  & = & \mrm{T}^{\pm}(u^{\eps}(\tau)-w^+)  & \mbox{ on }\Sigma_{\pm L}.
\end{array}
\end{equation}
In (\ref{FormulationRef}), $\mrm{T}^{\pm}$ are the standard Dirichlet-to-Neumann operators on $\Sigma_{\pm L}$. With the Riesz representation theorem, we introduce the bounded operator $\mrm{A}^{\eps}(\tau):\mH^1(\Om^{\eps}_L(\tau))\to\mH^1(\Om^{\eps}_L(\tau))$ and the function $f^{\eps}(\tau)\in\mH^1(\Om^{\eps}_L(\tau))$, such that, for all $\varphi$, $\varphi'\in\mH^1(\Om^{\eps}_L(\tau))$,
\begin{equation}\label{defOpA}
\hspace{-0.2cm}\begin{array}{|ll}
(\mrm{A}^{\eps}(\tau)\varphi,\varphi')_{\mH^1(\Om^{\eps}_L(\tau))}=\dsp\int_{\Om^{\eps}_L(\tau)} \nabla\varphi\cdot\overline{\nabla\varphi'}\,dz-k^2\int_{\Om^{\eps}_L(\tau)} \varphi\,\overline{\varphi'}\,dz-\langle\mrm{T}^{+}\varphi,\overline{\varphi'}\rangle_{\Sigma_{L}}-\langle\mrm{T}^{-}\varphi,\overline{\varphi'}\rangle_{\Sigma_{-L}} \\[15pt]
(f^{\eps}(\tau),\varphi')_{\mH^1(\Om^{\eps}_L(\tau))}=\langle\partial_{\nu} w^+-\mrm{T}^{+}w^+,\overline{\varphi'}\rangle_{\Sigma_{L}}+\langle\partial_{\nu} w^+-\mrm{T}^{-}w^+,\overline{\varphi'}\rangle_{\Sigma_{-L}}. 
\end{array}
\end{equation}
Here, $\langle\cdot,\cdot\rangle_{\Sigma_{\pm L}}$ denotes the duality pairing $\mH^{-1/2}(\Sigma_{\pm L})\times\mH^{1/2}(\Sigma_{\pm L})$. The function $u^{\eps}(\tau)\in\mH^1(\Om^{\eps}_L(\tau))$ is a solution of (\ref{FormulationRef}) if and only if it satisfies 
\begin{equation}\label{ReformulationRiesz}
\begin{array}{|rcll}
\multicolumn{4}{|l}{\mbox{Find }u^{\eps}(\tau)\in\mrm{H}^1(\Om^{\eps}_L(\tau)) \mbox{ such that } }\\[3pt]
\mrm{A}^{\eps}(\tau)u^{\eps}(\tau)=f^{\eps}(\tau).
\end{array}
\end{equation}
For any given $\tau\in(-\pi/(2k);\pi/(2k))^3$, following for example the proof of Lemma 3 in \cite{CDJT06} (again see also the previous studies \cite{Beal73,Gady93,KoMM94,Naza96}) and working by contradiction, we can prove that $\mrm{A}^{\eps}(\tau)$ is invertible for $\eps$ small enough with the stability estimate
\begin{equation}\label{StabilityEstimate}
\|(\mrm{A}^{\eps}(\tau))^{-1}\| \le C_{\tau}.
\end{equation}
In (\ref{StabilityEstimate}), $\|\cdot\|$ refers to the usual norm for linear operators of $\mrm{H}^1(\Om^{\eps}_L)$ and $C_{\tau}>0$ is a constant independent of $\eps$ (\textit{a priori} it may depend on $\tau$). 

\subsection{Error estimate for $\tau=0$}

We first establish error estimates with respect to $\eps$ for $\tau=0$. To shorten notation, we shall write $\mrm{A}^{\eps}$, $u^{\eps}$ instead of $\mrm{A}^{\eps}(0)$, $u^{\eps}(0)$. In general, the asymptotic expansion of $u^{\eps}$ constructed in \S\ref{describExpansion} (see (\ref{DefAnsatz})) is not in $\mH^1(\Om^{\eps})$ due to a jump of traces through $\Upsilon^{\eps}_m$ of order $\eps^2$. This is a drawback to establish error estimates and leads us to propose another approximation in $\Om^0$. For $m=1,2,3$, introduce $\hat{\zeta}_m\in\mathscr{C}^{\infty}(\R^2,[0;1])$ a cut-off function such that $\hat{\zeta}_m=1$ for $|z-M_m|\le1$ and $\hat{\zeta}_m=0$ for $|z-M_m|\ge2$. Set $\zeta^{\eps}_m(z):=\hat{\zeta}_m(z/\eps)$ and  $\zeta^{\eps}=1-\zeta^{\eps}_1-\zeta^{\eps}_2-\zeta^{\eps}_3$. Finally define $u^{\eps}_{\mrm{as}}$ such that
\begin{equation}\label{DefAnsatzBis1}
\begin{array}{lcl}
u^{\eps}_{\mrm{as}}(z)= \zeta^{\eps}(z)(u^{0}(z)+\eps\,u^{1}(z))
& + &\dsp\sum_{m=1}^3\zeta^{\eps}_m(z)(u^{0}(M_m)+(x-x_m)\partial_x u^{0}(M_m)+\eps\,u^{1}(M_m))\\[10pt]
& + &\dsp\sum_{m=1}^3\zeta_m(z)(\eps\,(\ln\eps\,a_m+V^{1}_m(\xi^m))),\qquad z=(x,y)\in\Om^0,
\end{array}
\end{equation}
and, for $m=1,2,3$,
\begin{equation}\label{DefAnsatzBis2}
u^{\eps}_{\mrm{as}}(z)= v^{0}_m(y)+\zeta_m(z)(\eps\,(\ln\eps\,a_m+V^{1}_m(\xi^m))),\qquad z=(x,y)\in\mathscr{S}^{\eps}_m.
\end{equation}
Here $u^{0}$, $u^{1}$, $v^0_m$, $V^{1}_m$, $\zeta_m$ are the same terms as in (\ref{DefAnsatz}) while $a_m$ are the constants defined in (\ref{defAm}). Note that for $z=(x,y)$ such that $|x|\ge L/2$, we have $(\ref{DefAnsatzBis1})=(\ref{DefAnsatz})=u^{0}+\eps\,u^{1}$. As a consequence, the two approximations (\ref{DefAnsatzBis1}) and (\ref{DefAnsatz})
yield the same first term $s^{1\pm}$ (see  (\ref{expressionReflecCoeffterThree})) in the asymptotic expansion of $s^{\eps\pm}$. 
\begin{proposition}\label{PropositionErrorEstimate}
There is $\eps_0>0$ such that for all $\eps\in(0;\eps_0]$ we have
\begin{equation}\label{errorEstimate}
\|u^{\eps}-u^{\eps}_{\mrm{as}}\|_{\mH^1(\Om^{\eps}_L)} \le C\,\eps^{3/2}(1+|\ln\eps|).
\end{equation}
\end{proposition}

\begin{proof}
Inequality (\ref{StabilityEstimate}) with $\tau=0$ already provides a stability estimate. Therefore, to obtain (\ref{errorEstimate}), it just remains to compute the consistency error
\begin{equation}\label{defSup}
\|\mrm{A}^{\eps}(u^{\eps}-u^{\eps}_{\mrm{as}})\|_{\mH^1(\Om^{\eps}_L)} = \underset{v^{\eps}\in\mH^1(\Om^{\eps}_L)\,|\,\|v^{\eps}\|_{\mH^1(\Om^{\eps}_L)}=1}{\sup}\ |(\mrm{A}^{\eps}(u^{\eps}-u^{\eps}_{\mrm{as}}),v^{\eps})_{\mH^1(\Om^{\eps}_L)}|.
\end{equation}
We have $(\mrm{A}^{\eps}u^{\eps},v^{\eps})_{\mH^1(\Om^{\eps}_L)}=(f^{\eps},v^{\eps})_{\mH^1(\Om^{\eps}_L)}=\langle\partial_{\nu} w^+-\mrm{T}^{+}w^+,\overline{v^{\eps}}\rangle_{\Sigma_{L}}+\langle\partial_{\nu} w^+-\mrm{T}^{-}w^+,\overline{v^{\eps}}\rangle_{\Sigma_{-L}}$. From definition (\ref{defOpA}) of $\mrm{A}^{\eps}$, we deduce 
\[
\begin{array}{lcl}
(\mrm{A}^{\eps}(u^{\eps}-u^{\eps}_{\mrm{as}}),v^{\eps})_{\mH^1(\Om^{\eps}_L)} & = & -\dsp\int_{\Om^{\eps}_L} \nabla u^{\eps}_{\mrm{as}}\cdot\overline{\nabla v^{\eps}}\,dz+k^2\int_{\Om^{\eps}_L} u^{\eps}_{\mrm{as}}\,\overline{v^{\eps}}\,dz\\[10pt]
& & +\langle\partial_{\nu} w^++\mrm{T}^{+}(u^{\eps}_{\mrm{as}}-w^+),\overline{v^{\eps}}\rangle_{\Sigma_{L}}+\langle\partial_{\nu} w^++\mrm{T}^{-}(u^{\eps}_{\mrm{as}}-w^+),\overline{v^{\eps}}\rangle_{\Sigma_{-L}}.
\end{array}
\]
Note that $\partial_yu^{\eps}_{\mrm{as}}(x,1^-)=\partial_yu^{\eps}_{\mrm{as}}(x,1^+)$ for $x\in(-\eps/2+x_m;x_m+\eps/2)$ and that $\partial_{\nu}u^{\eps}_{\mrm{as}}=0$ on $\partial\Om^{\eps}\cap\partial\Om^{\eps}_L$. Since $u^{\eps}_{\mrm{as}}=u^0+\eps\,u^1=w^++\eps\,u^1$ on $\Sigma_{\pm L}$ where $u^1$ is outgoing, integrating by parts yields
\begin{equation}\label{CalculSup}
(\mrm{A}^{\eps}(u^{\eps}-u^{\eps}_{\mrm{as}}),v^{\eps})_{\mH^1(\Om^{\eps}_L)}  = \dsp\int_{\Om^{\eps}_L} (\Delta u^{\eps}_{\mrm{as}}+k^2 u^{\eps}_{\mrm{as}})\,\overline{v^{\eps}}\,dz.
\end{equation}
If $\zeta$, $\varphi$ are two functions, we define the commutator $[\Delta,\zeta](\varphi):=\Delta(\zeta \varphi)-\zeta\Delta \varphi=2\nabla\varphi\cdot\nabla \zeta+ \varphi\Delta \zeta$. In $\Om^0_L$, we have
\begin{equation}\label{defBigRemainder}
\begin{array}{ll}
&\Delta u^{\eps}_{\mrm{as}}+k^2 u^{\eps}_{\mrm{as}}\\[2pt] 
 = & -\dsp\sum_{m=1}^3[\Delta,\zeta^{\eps}_m](\tilde{u}^0_m+\eps\tilde{u}^1_m) + \dsp k^2\sum_{m=1}^3\zeta^{\eps}_m(z)(u^{0}(M_m)+(x-x_m)\partial_x u^{0}(M_m)+\eps\,u^{1}(M_m))\\[10pt]
 &+ \eps\dsp\sum_{m=1}^3([\Delta,\zeta_m]+k^2\zeta_m)\tilde{V}^1_m(\eps^{-1}(z-M_m)) + \eps k^2\dsp\sum_{m=1}^3\zeta^{\eps}_m(z)a_m\ln r_m
\end{array}
\end{equation}
where $a_m$ is defined in (\ref{defAm}) and where 
\[
\begin{array}{lcl}
\tilde{u}^0_m=u^0-\left( u^0(M_m)+(x-x_m)\partial_x u^{0}(M_m)\right)=w^+-\left( w^+(M_m)+(x-x_m)\partial_x w^+(M_m)\right),\\[6pt]
\tilde{u}^1_m=u^1-u^1(M_m),\\[6pt]
\tilde{V}^1_m=V^1_m-a_m\ln(\eps^{-1} r_m).
\end{array}
\]
The next step consists in studying the contribution of each of the terms on the right hand side of (\ref{defBigRemainder}). To shorten the presentation, we will consider only the first one. The analysis to deal with the other terms is very similar. To proceed, we will use the following Hardy inequality with logarithm (see \cite{HaLP52}): for $v^{\eps}\in\mH^1(\Om^{\eps}_L)$ such that $\|v^{\eps}\|_{\mH^1(\Om^{\eps}_L)}=1$, denoting $r_m=|z-M_m|$, we have 
\begin{equation}\label{HardyIneg}
\|r_m^{-1}(1+|\ln r_m|)^{-1}v^{\eps}\|_{\mrm{L}^2(\Om^0_L)}\le C.
\end{equation}
Define the annulus $\mrm{Q}_{\eps}(M_m)=\mrm{D}_{2\eps}(M_m)\setminus\overline{\mrm{D}_{\eps}(M_m)}$. Observe that $[\Delta,\zeta^{\eps}_m](\tilde{u}^0_m+\eps\tilde{u}^1_m)$ is supported in $\overline{\mrm{Q}_{\eps}(M_m)}$.   Moreover, in this domain, we have $\tilde{u}^0_m(z)=O(r_m^2)$, $\nabla\tilde{u}^0_m(z)=O(r_m)$, $|\Delta\zeta^{\eps}_m|\le C\,\eps^{-2}$ and $|\nabla\zeta^{\eps}_m|\le C\,\eps^{-1}$. Using also the estimate\footnote{This (non-optimal) result can be obtained rigorously working in weighted Sobolev (Kondratiev) spaces \cite{Kond67}.} $\|r_m^{-1/2}\nabla\tilde{u}^1_m\|_{\mrm{L^2(\Om^0_L)}}+\|r_m^{-3/2}\tilde{u}^1_m\|_{\mrm{L^2(\Om^0_L)}}\le C$, we can write
\begin{equation}\label{FirstTermRhsTrav1}
\begin{array}{ll}
 & |([\Delta,\zeta^{\eps}_m](\tilde{u}^0_m+\eps\tilde{u}^1_m),v^{\eps})_{\Om^0_L}| \\[6pt]
 \le & \phantom{+}C\eps^{-2}\|r_m(1+|\ln r_m|)(\tilde{u}^0_m+\eps\tilde{u}^1_m)\|_{\mrm{L}^2(\mrm{Q}_{\eps}(M_m))}\|r_m^{-1}(1+|\ln r_m|)^{-1}v^{\eps}\|_{\mrm{L}^2(\mrm{Q}_{\eps}(M_m))}\\[6pt]
   & +C\eps^{-1}\|r_m(1+|\ln r_m|)\nabla(\tilde{u}^0_m+\eps\tilde{u}^1_m)\|_{\mrm{L}^2(\mrm{Q}_{\eps}(M_m))}\|r_m^{-1}(1+|\ln r_m|)^{-1}v^{\eps}\|_{\mrm{L}^2(\mrm{Q}_{\eps}(M_m))}\\[6pt]
 \le & C\,\eps^{3/2}(1+|\ln\eps|).   
\end{array}
\end{equation}
Proceeding similarly to study all the terms on the right hand side of (\ref{defBigRemainder}), we find 
\begin{equation}\label{EstimRemainder1}
|(\Delta u^{\eps}_{\mrm{as}}+k^2 u^{\eps}_{\mrm{as}},v^{\eps})_{\Om^0_L}|\le  C\,\eps^{3/2}(1+|\ln\eps|). 
\end{equation}
In $\mathscr{S}^{\eps}_m$, for $m=1,2,3$, we have $\Delta u^{\eps}_{\mrm{as}}+k^2 u^{\eps}_{\mrm{as}} = \tilde{f}^1_m$ with 
\[
\tilde{f}^1_m(z) = \eps ([\Delta_z,\zeta_m]+k^2\zeta_m)(\ln\eps\,a_m+V^1_m(\eps^{-1}(z-M_m)).
\]
Observing that $|\tilde{f}^1_m(z)|\le C\,\eps(1+|\ln\eps|)$ for $z=(x,y)\in \mathscr{S}^{\eps}_m$, we can write
\begin{equation}\label{EstimRemainder2}
|(\Delta u^{\eps}_{\mrm{as}}+k^2 u^{\eps}_{\mrm{as}},v^{\eps})_{\mathscr{S}^{\eps}_m}| \le \|\Delta u^{\eps}_{\mrm{as}}+k^2 u^{\eps}_{\mrm{as}}\|_{\mrm{L}^2(\mathscr{S}^{\eps}_m)}\,\|v^{\eps}\|_{\mrm{L}^2(\mathscr{S}^{\eps}_m)} \ \le \ C\,\eps^{3/2}(1+|\ln\eps|).
\end{equation}
Using (\ref{EstimRemainder1}), (\ref{EstimRemainder2}) in (\ref{defSup}), (\ref{CalculSup}), we find $\|\mrm{A}^{\eps}(u^{\eps}-u^{\eps}_{\mrm{as}})\|_{\mH^1(\Om^{\eps}_L)} \le C\,\eps^{3/2}(1+|\ln\eps|)$. With the stability estimate (\ref{StabilityEstimate}), we deduce $\|u^{\eps}-u^{\eps}_{\mrm{as}}\|_{\mH^1(\Om^{\eps}_L)} \le C\,\eps^{3/2}(1+|\ln\eps|)$.
\end{proof}

\subsection{Error estimates with respect to $\tau$}\label{paragraphDefiRemainder}

Now, we have all the tools to establish estimate (\ref{MainEstimate}). In the following, as in the classical proofs of perturbations theory for linear operators (see \cite[Chapter 7, \S6.5]{Kato95}, \cite[Chap. 4]{HiPh57}), we will use a change of variables to compare the solutions of (\ref{ReformulationRiesz}) in the same geometry $\Om^\eps_L(0)=\Om^\eps_L$. To proceed, introduce some smooth diffeomorphism $\mathscr{L}(\tau)$ which maps $\Om^{\eps}_L$ into $\Om^{\eps}_L(\tau)$. Note that we can take $\mathscr{L}(\tau)$ acting only on the $y$ variable and therefore, independent of $\eps$. We can also assume that the global change of variables is identical in a neighbourhood of $\overline{\Om^0_L}$. \\
\newline
For $\tau\in(-\pi/(2k);\pi/(2k))^3$, define the function $u^{\eps}_{\mrm{as}}(\tau)$ as in (\ref{DefAnsatzBis1}), (\ref{DefAnsatzBis2}) where the terms $u^1$, $V^1_m$, $v^1_m$ are computed in \S\ref{describExpansion} with $h_m$ replaced by $h_m+\tau_m$. Set $\tilde{u}^{\eps}(\tau)=u^{\eps}(\tau)-u^{\eps}_{\mrm{as}}(\tau)$. With the Riesz representation theorem, introduce the function $\tilde{f}^{\eps}(\tau)\in\mH^1(\Om^{\eps}_L(\tau))$, such that, for all $\varphi'\in\mH^1(\Om^{\eps}_L(\tau))$,
\[
(\tilde{f}^{\eps}(\tau),\varphi')_{\mH^1(\Om^{\eps}_L(\tau))}=\dsp\int_{\Om^{\eps}_L(\tau)} (\Delta u^{\eps}_{\mrm{as}}(\tau)+k^2 u^{\eps}_{\mrm{as}}(\tau))\,\overline{\varphi'}\,dz.
\]
According to (\ref{CalculSup}), the remainder $\tilde{u}^{\eps}(\tau)$ satisfies $\mrm{A}^{\eps}(\tau)\tilde{u}^{\eps}(\tau)=\tilde{f}^{\eps}(\tau)$ in $\mrm{H}^1(\Om^{\eps}_L(\tau))$. Define $\tilde{U}^{\eps}(\tau):=\tilde{u}^{\eps}(\tau)\circ \mathscr{L}(\tau)$ and  $\tilde{F}^{\eps}(\tau):=\tilde{f}^{\eps}(\tau)\circ \mathscr{L}(\tau)$.  We have 
\begin{equation}\label{line1}
\|\tilde{F}^{\eps}(\tau)-\tilde{F}^{\eps}(\tau')\|_{\mrm{H}^1(\Om^{\eps}_L)} \le  C\,\eps^{3/2}(1+|\ln\eps|)\,|\tau-\tau'|
\end{equation}
for all $\eps\in(0;\eps_{0}]$, $\tau$, $\tau'\in[-\Theta;\Theta]^3$, where $\Theta$ is set in $(0;\pi/(2k))$. On the other hand, under the change of variables $\mathscr{L}(\tau)$, $\mrm{A}^{\eps}(\tau):  \mrm{H}^1(\Om^{\eps}_L(\tau)) \rightarrow  \mrm{H}^1(\Om^{\eps}_L(\tau))$ is transformed into the operator $\mathscr{A}^{\eps}(\tau): \mrm{H}^1(\Om^{\eps}_L) \rightarrow  \mrm{H}^1(\Om^{\eps}_L)$. In particular, there holds 
\begin{equation}\label{EqnEquiv}
\mrm{A}^{\eps}(\tau)\tilde{u}^{\eps}(\tau)=\tilde{f}^{\eps}(\tau) \qquad \Leftrightarrow\qquad\mathscr{A}^{\eps}(\tau)\tilde{U}^{\eps}(\tau)=\tilde{F}^{\eps}(\tau).
\end{equation}
Moreover, the coefficients of $\mathscr{A}^{\eps}(\tau)$ depend smoothly on the parameter $\tau\in[-\Theta;\Theta]^3$. Therefore, we have the estimate 
\begin{equation}\label{LipschitzOperator}
\|\mathscr{A}^{\eps}(\tau)-\mathscr{A}^{\eps}(\tau')\|\le C\,|\tau-\tau'|,\qquad\forall\eps\in(0;\eps_{0}],\ \tau,\,\tau'\in[-\Theta;\Theta]^3.
\end{equation}
For $\tau\in[-\Theta;\Theta]^3$, we can write
\[
\mathscr{A}^{\eps}(\tau)=\mathscr{A}^{\eps}(0)(\mrm{Id}+(\mathscr{A}^{\eps}(0))^{-1}(\mathscr{A}^{\eps}(\tau)-\mathscr{A}^{\eps}(0))).
\]
Since $(\mathscr{A}^{\eps}(0))^{-1}=(\mrm{A}^{\eps}(0))^{-1}$ is uniformly bounded for $\eps\in(0;\eps_{0}]$ (inequality (\ref{StabilityEstimate}) with $\tau=0$), we deduce that there is some $\vartheta\in(0;\pi/(2k))$ such that $\mathscr{A}^{\eps}(\tau)$ is uniformly invertible for all $\eps\in(0;\eps_{\vartheta}],\ \tau,\,\tau'\in[-\vartheta;\vartheta]^3$. In particular, this implies 
\begin{equation}\label{line0}
\|\tilde{U}^{\eps}(\tau)\|_{\mrm{H}^1(\Om^{\eps}_L)} \le  C\,\eps^{3/2}(1+|\ln\eps|),\qquad\forall\eps\in(0;\eps_{\vartheta}],\ \tau,\,\tau'\in[-\vartheta;\vartheta]^3.
\end{equation}
Then, from 
\[
\begin{array}{ll}
 & \mathscr{A}^{\eps}(\tau)\tilde{U}^{\eps}(\tau)-\mathscr{A}^{\eps}(\tau')\tilde{U}^{\eps}(\tau') = \tilde{F}^{\eps}(\tau)-\tilde{F}^{\eps}(\tau') \\[6pt]
\Leftrightarrow & \mathscr{A}^{\eps}(\tau)(\tilde{U}^{\eps}(\tau)-\tilde{U}^{\eps}(\tau'))
=(\mathscr{A}^{\eps}(\tau')-\mathscr{A}^{\eps}(\tau))\tilde{U}^{\eps}(\tau') + \tilde{F}^{\eps}(\tau)-\tilde{F}^{\eps}(\tau'),
\end{array}
\]
together with (\ref{line0}), (\ref{line1}) and (\ref{LipschitzOperator}), we obtain 
\[
\|\tilde{U}^{\eps}(\tau)-\tilde{U}^{\eps}(\tau')\|_{\mH^1(\Om^{\eps}_L)} \le C\,\eps^{3/2}(1+|\ln\eps|)\,|\tau-\tau'|,\qquad\forall\eps\in(0;\eps_{\vartheta}],\ \tau,\,\tau'\in[-\vartheta;\vartheta]^3.
\] 
Finally, remarking that $\tilde{U}^{\eps}(\tau)=\tilde{u}^{\eps}(\tau)$ and $\tilde{U}^{\eps}(\tau')=\tilde{u}^{\eps}(\tau')$ in $\Om^0$ (we remind the reader that $\mathscr{L}(\tau)$ is the identity in $\Om^0_L$), we infer
\[
\begin{array}{lcl}
|\tilde{s}^{\eps\pm}(\tau)-\tilde{s}^{\eps\pm}(\tau')| & = & \left|\dsp\int_{\Sigma_{-L}\cup\Sigma_{+L}} \frac{\partial (\tilde{u}^{\eps}(\tau)-\tilde{u}^{\eps}(\tau'))}{\partial\nu}\,\overline{w^{\pm}}- (\tilde{u}^{\eps}(\tau)-\tilde{u}^{\eps}(\tau'))\frac{\partial\overline{w^{\pm}}}{\partial\nu}\,d\sigma\right|\\[18pt]
& = &  \left|\dsp\int_{\Sigma_{-L}\cup\Sigma_{+L}} \frac{\partial (\tilde{U}^{\eps}(\tau)-\tilde{U}^{\eps}(\tau'))}{\partial\nu}\,\overline{w^{\pm}}- (\tilde{U}^{\eps}(\tau)-\tilde{U}^{\eps}(\tau'))\frac{\partial\overline{w^{\pm}}}{\partial\nu}\,d\sigma\right|\\[13pt]
 & \le & C\,\eps^{3/2}(1+|\ln\eps|)\,|\tau-\tau'|.
\end{array}
\]
We emphasize that the constant $C>0$ appearing in the last inequality above is independent of $\eps\in(0;\eps_{\vartheta}]$, $\tau$, $\tau'\in[-\vartheta;\vartheta]^3$. This ends to prove estimate (\ref{MainEstimate}).

\section{Conclusion and discussion}\label{SectionConclu}

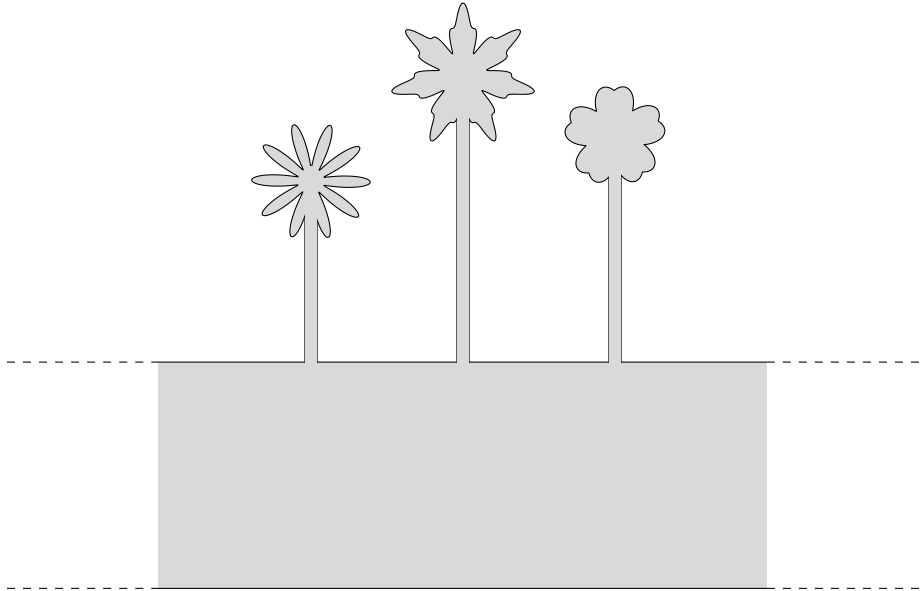
\begin{figure}[!ht]
\centering
\begin{tikzpicture}[scale=2]
\draw[fill=gray!30,draw=none](-2,0.5) rectangle (2,2);
\draw (-2,0.5)--(2,0.5); 
\draw[dashed] (-3,0.5)--(-2,0.5); 
\draw[dashed] (3,0.5)--(2,0.5); 
\draw (-2,2)--(2,2);  
\draw[dashed] (3,2)--(2,2);
\draw[dashed] (-3,2)--(-2,2);
\draw[fill=gray!30](-0.04,2) rectangle (0.04,3.8);
\draw[fill=gray!30](-0.96,2) rectangle (-1.04,3.3);
\draw[fill=gray!30](0.96,2) rectangle (1.04,3.5);
\begin{scope}[shift={(0,3.9)},scale=0.2]
\draw [domain=0:360, black, samples=300, smooth,fill=gray!30]
plot ({cos(\x)*(1.6-0.8*sin(7*\x)^3)},{sin(\x)*(1.6-0.8*sin(7*\x)^3)});
\end{scope}
\begin{scope}[shift={(-1,3.2)},scale=0.3]
\draw [domain=0:360, black, samples=300, smooth,fill=gray!30]
plot (\x:{0.8+sin(5*(\x-26))*cos(5*(\x-26))});
\end{scope}
\begin{scope}[shift={(1,3.5)},scale=0.1]
\draw [domain=0:360, black, samples=300, smooth,fill=gray!30]
plot (\x:{2+abs(sin(2.5*(\x-55))*cos(2.5*(\x-55)))+abs(sin(2.5*(\x-55)))});
\end{scope}
\draw[fill=gray!30,draw=none](-0.037,1.9) rectangle (0.037,3.8);
\draw[fill=gray!30,draw=none](-0.963,1.9) rectangle (-1.037,3.3);
\draw[fill=gray!30,draw=none](0.963,1.9) rectangle (1.037,3.5);
\end{tikzpicture}
\caption{Invisible garden of flowers.\label{InvisibleGardenOfFlowers}} 
\end{figure}

\noindent In the first part of the article, we showed that for a given perturbed waveguide, transmission invisibility ($T=1$), cannot be achieved for wavenumbers smaller than a given bound $k_{\star}$ (see Proposition \ref{MainProposition}). This parameter $k_{\star}$ depends on the geometry of the perturbation. In particular, if the perturbation is smooth and small (in amplitude and in width), $k_{\star}$ is very close to the threshold wavenumber $\lambda_1$ (equal to $\pi^2$ when the reference waveguide is $\R\times(0;1)$). In the second part of the article, for any given wavenumber $k$ such that $0<k^2<\lambda_1$, we proposed a method to construct waveguides where there holds $R=0$ and $T=1$ (Proposition \ref{propositionMainResult}). Therefore, in these waveguides, the incident wave $w^+$ goes through and produces a scattered field which is exponentially decaying at $\pm\infty$. To proceed, and to circumvent the obstruction highlighted in the first part of the paper, we worked with singular perturbations of the geometry. More precisely, we added thin rectangles to the reference waveguide and we played with their height (parametrized by the vector $\tau$). Similarly, one could have tuned the parameter $x_m$ (the abscissa of the rectangles) looking for $x_m$ under the form $x_m(\tau)=x_m(0)+\tau_m$. Note also that we played with thin rectangles only to simplify as far as possible the analysis. Using the same method and tuning the height of the stalks, one could construct invisible gardens of flowers as in Figure \ref{InvisibleGardenOfFlowers}. The only assumption which has to be satisfied is that $k^2$ is not a resonance frequency of the flowers (in our approach, this condition was used just after (\ref{Pb1D})).\\
\newline 
One possible direction to continue the analysis is to consider the multi-modal case, that is a setting where the wavenumber is chosen so that several modes can propagate in the waveguide. In this situation, instead of having to cancel two coefficients $s^{\pm}$, one has to cancel two matrices $(s^{\pm}_{mn})_{0\le m,n \le N-1}$ (note that $(s^{-}_{mn})_{0\le m,n \le N-1}$ is symmetric), where $N$ is the number of propagative incident waves. Of course, for such a problem it is necessary to introduce new degrees of freedom playing with more than three flowers. Adapting the technique presented in \cite[\S8.2]{ChNa16}, one may achieve to impose $s^{-}_{mn}=0$ for $m,n=0,\dots,N-1$ except for a discrete set of wavenumbers where the technique fails (the reason being that the equivalent of the matrix $\mathbb{M}$ defined in (\ref{defMatrixM}) is not invertible). It would be interesting to understand if for these pathological configurations there is a real obstruction to non reflectivity or, if this is only our approach which is inefficient. Imposing $s^{+}_{mn}=0$ for $m,n=0,\dots,N-1$ when several modes can propagate in the waveguide seems to be an even more challenging problem. In that situation, for all wavenumbers, the equivalent of the matrix $\mathbb{M}$ is not invertible.

\section*{Acknowledgments}
The research of L. C. was supported by the FMJH through the grant ANR-10-CAMP-0151-02 in the ``Programme des Investissements
d'Avenir''. The work of S.A. N. was supported by the grant No. 15-01-02175 of Russian Foundation on Basic Research as well as by the LabEx LMH, through the grant ANR-11-LABX-0056-LMH in the ``Programme des Investissements d'Avenir''.

\bibliography{Bibli}

\def\cprime{$'$}
\begin{thebibliography}{10}

\bibitem{AlSE08}
A.~Al{\`u}, {M.G.} Silveirinha, and N.~Engheta.
\newblock Transmission-line analysis of $\varepsilon$-near-zero--filled narrow
  channels.
\newblock {\em Phys. Rev. E}, 78(1):016604, 2008.

\bibitem{ArGL11}
T.~Arens, D.~Gintides, and A.~Lechleiter.
\newblock Direct and inverse medium scattering in a three-dimensional
  homogeneous planar waveguide.
\newblock {\em SIAM J. Appl. Math.}, 71(3):753--772, 2011.

\bibitem{Arse76}
{A.A.} Arsen'ev.
\newblock The existence of resonance poles and scattering resonances in the
  case of boundary conditions of the second and third kind.
\newblock {\em USSR Comput. Math. Math. Phys.}, 16(3):171--177, 1976.

\bibitem{BaNa15}
{F.L.} Bakharev and {S.A.} Nazarov.
\newblock Gaps in the spectrum of a waveguide composed of domains with
  different limiting dimensions.
\newblock {\em Sib. Math. J.}, 56(4):575--592, 2015.

\bibitem{Beal73}
{J.T.} Beale.
\newblock Scattering frequencies of resonators.
\newblock {\em Comm. Pure Appl. Math.}, 26(4):549--563, 1973.

\bibitem{BlPS14}
E.~Bl{\aa}sten, L.~P{\"a}iv{\"a}rinta, and J.~Sylvester.
\newblock Corners always scatter.
\newblock {\em Commun. Math. Phys.}, 331(2):725--753, 2014.

\bibitem{BoCNSu}
{A.-S.} {Bonnet-Ben Dhia}, {L.} Chesnel, and {S.A.} Nazarov.
\newblock Non-scattering wavenumbers and far field invisibility for a finite
  set of incident/scattering directions.
\newblock {\em Inverse Problems}, 31(4):045006, 2015.

\bibitem{BLMN17}
{A.-S.} Bonnet-Ben~Dhia, E.~Lun{\'e}ville, Y.~Mbeutcha, and {S.A.} Nazarov.
\newblock A method to build non-scattering perturbations of two-dimensional
  acoustic waveguides.
\newblock {\em Math. Methods Appl. Sci.}, 40(2):335--349, 2017.

\bibitem{BoNa13}
{A.-S.} {Bonnet-Ben Dhia} and {S.A.} Nazarov.
\newblock Obstacles in acoustic waveguides becoming ``invisible'' at given
  frequencies.
\newblock {\em Acoust. Phys.}, 59(6):633--639, 2013.

\bibitem{BoNTSu}
{A.-S.} {Bonnet-Ben Dhia}, {S.A.} Nazarov, and {J.} Taskinen.
\newblock Underwater topography ``invisible'' for surface waves at given
  frequencies.
\newblock {\em Wave Motion}, 57(0):129--142, 2015.

\bibitem{BoNg15}
L.~Borcea and {D.-L.} Nguyen.
\newblock Imaging with electromagnetic waves in terminating waveguides.
\newblock {\em Inverse Problems and Imaging}, 10(4):915--941, 2016.

\bibitem{BoFl14}
L.~Bourgeois and S.~Fliss.
\newblock On the identification of defects in a periodic waveguide from far
  field data.
\newblock {\em Inverse Problems}, 30(9):095004, 2014.

\bibitem{BoLu08}
L.~Bourgeois and E.~Lun{\'e}ville.
\newblock The linear sampling method in a waveguide: a modal formulation.
\newblock {\em Inverse problems}, 24(1):015018, 2008.

\bibitem{CaNR12}
G.~Cardone, {S.A.} Nazarov, and K.~Ruotsalainen.
\newblock Asymptotic behaviour of an eigenvalue in the continuous spectrum of a
  narrowed waveguide.
\newblock {\em Sb. Math.}, 203(2):153, 2012.

\bibitem{ChHS15}
L.~Chesnel, N.~Hyv{\"o}nen, and S.~Staboulis.
\newblock Construction of indistinguishable conductivity perturbations for the
  point electrode model in electrical impedance tomography.
\newblock {\em SIAM J. Appl. Math.}, 75(5):2093--2109, 2015.

\bibitem{ChNa16}
L.~Chesnel and {S.A.} Nazarov.
\newblock Team organization may help swarms of flies to become invisible.
\newblock {\em Inverse Problems and Imaging}, 10(4):977--1006, 2016.

\bibitem{CDJT06}
M.~Clausel, M.~Durufl{\'e}, P.~Joly, and S.~Tordeux.
\newblock A mathematical analysis of the resonance of the finite thin slots.
\newblock {\em Appl. Numer. Math.}, 56(10):1432--1449, 2006.

\bibitem{EASE09}
B.~Edwards, A.~Al{\`u}, {M.G.} Silveirinha, and N.~Engheta.
\newblock Reflectionless sharp bends and corners in waveguides using
  epsilon-near-zero effects.
\newblock {\em J. Appl. Phys.}, 105(4):044905, 2009.

\bibitem{ElHu15}
J.~Elschner and G.~Hu.
\newblock Corners and edges always scatter.
\newblock {\em Inverse Problems}, 31(1):015003, 2015.

\bibitem{EvMP14}
{D.V.} Evans, M.~McIver, and R.~Porter.
\newblock Transparency of structures in water waves.
\newblock In {\em Proceedings of 29th International Workshop on Water Waves and
  Floating Bodies}, 2014.

\bibitem{FlAl13}
R.~Fleury and A.~Al{\`u}.
\newblock Extraordinary sound transmission through density-near-zero
  ultranarrow channels.
\newblock {\em Phys. Rev. Lett.}, 111(5):055501, 2013.

\bibitem{FuXC14}
Y.~Fu, Y.~Xu, and H.~Chen.
\newblock Additional modes in a waveguide system of zero-index-metamaterials
  with defects.
\newblock {\em Scientific reports}, 4, 2014.

\bibitem{Gady93}
{R.R.} Gadyl'shin.
\newblock {Characteristic frequencies of bodies with thin spikes. I.
  Convergence and estimates}.
\newblock {\em Math. Notes}, 54(6):1192--1199, 1993.

\bibitem{Gady05}
{R.R.} Gadyl'shin.
\newblock {On the eigenvalues of a ``dumbbell with a thin handle''.}
\newblock {\em Izv. Math.}, 69(2):265--329, 2005.

\bibitem{HaLP52}
G.H Hardy, J.E Littlewood, and G.~P{\'o}lya.
\newblock {\em Inequalities. 2nd ed.}
\newblock Cambridge university press, 1952.

\bibitem{Henr06}
A.~Henrot.
\newblock {\em Extremum problems for eigenvalues of elliptic operators}.
\newblock {Birkh\"{a}user}, 2006.

\bibitem{HiPh57}
E.~Hille and {R.S.} Phillips.
\newblock {\em Functional analysis and semi-groups}, volume~31.
\newblock Amer. Math. Soc., 1957.

\bibitem{HuSV16}
G.~Hu, M.~Salo, and {E.V.} Vesalainen.
\newblock Shape identification in inverse medium scattering problems with a
  single far-field pattern.
\newblock {\em SIAM J. Math. Anal.}, 48(1):152--165, 2016.

\bibitem{Jack99}
{J.D.} Jackson.
\newblock {\em Classical electrodynamics, Third Edition}.
\newblock Wiley, New York, 1999.

\bibitem{JoTo06b}
P.~Joly and S.~Tordeux.
\newblock Asymptotic analysis of an approximate model for time harmonic waves
  in media with thin slots.
\newblock {\em Math. Mod. Num. Anal.}, 40(1):63--97, 2006.

\bibitem{JoTo06}
P.~Joly and S.~Tordeux.
\newblock {Matching of asymptotic expansions for wave propagation in media with
  thin slots I: The asymptotic expansion}.
\newblock {\em SIAM Multiscale Model. Simul.}, 5(1):304--336, 2006.

\bibitem{JoTo08}
P.~Joly and S.~Tordeux.
\newblock {Matching of asymptotic expansions for waves propagation in media
  with thin slots II: The error estimates}.
\newblock {\em Math. Mod. Num. Anal.}, 42(2):193--221, 2008.

\bibitem{Kato95}
T.~Kato.
\newblock {\em {Perturbation theory for linear operators.}}
\newblock Springer-Verlag, Berlin, reprint of the corr. print. of the 2nd ed.
  1980 edition, 1995.

\bibitem{Kond67}
{V.A.} {Kondratiev}.
\newblock Boundary-value problems for elliptic equations in domains with
  conical or angular points.
\newblock {\em Trans. Moscow Math. Soc.}, 16:227--313, 1967.

\bibitem{KoNS16}
{A.I.} Korolkov, {S.A.} Nazarov, and {A.V.} Shanin.
\newblock Stabilizing solutions at thresholds of the continuous spectrum and
  anomalous transmission of waves.
\newblock {\em Z. Angew. Math. Mech.}, 96(10):1245--1260, 2016.

\bibitem{KoMM94}
{V.A.} Kozlov, {V.G.} Maz'ya, and {A.B.} Movchan.
\newblock Asymptotic analysis of a mixed boundary value problem in a
  multi-structure.
\newblock {\em Asymptot. Anal.}, 8(2):105--143, 1994.

\bibitem{Mans12}
M.~Mansuripur.
\newblock New perspective on the optical theorem of classical electrodynamics.
\newblock {\em Am. J. Phys.}, 80(4):329--333, 2012.

\bibitem{MoSK12}
P.~Monk, V.~Selgas, and A.~Kirsch.
\newblock Sampling type methods for an inverse waveguide problem.
\newblock {\em Inverse Problems and Imaging}, 6:709--747, 2012.

\bibitem{Naza96}
{S.A.} Nazarov.
\newblock Junctions of singularly degenerating domains with different limit
  dimensions 1.
\newblock {\em J. Math. Sci. (N.Y.)}, 80(5):1989--2034, 1996.

\bibitem{Naza05}
{S.A.} Nazarov.
\newblock Asymptotic analysis and modeling of the jointing of a massive body
  with thin rods.
\newblock {\em J. Math. Sci. (N.Y.)}, 127(5):2192--2262, 2005.

\bibitem{Naza11}
{S.A.} {Nazarov}.
\newblock {Asymptotic expansions of eigenvalues in the continuous spectrum of a
  regularly perturbed quantum waveguide}.
\newblock {\em {Theor. Math. Phys.}}, 167(2):606--627, 2011.

\bibitem{Naza11c}
{S.A.} Nazarov.
\newblock {Eigenvalues of the laplace operator with the Neumann conditions at
  regular perturbed walls of a waveguide}.
\newblock {\em J. Math. Sci.}, 172(4):555--588, 2011.

\bibitem{Naza11b}
{S.A.} {Nazarov}.
\newblock Trapped waves in a cranked waveguide with hard walls.
\newblock {\em Acoust. Phys.}, 57(6):764--771, 2011.

\bibitem{Naza12}
{S.A.} Nazarov.
\newblock Enforced stability of an eigenvalue in the continuous spectrum of a
  waveguide with an obstacle.
\newblock {\em Comput. Math. and Math. Phys.}, 52(3):448--464, 2012.

\bibitem{Naza13}
{S.A.} Nazarov.
\newblock Enforced stability of a simple eigenvalue in the continuous spectrum
  of a waveguide.
\newblock {\em Funct. Anal. Appl.}, 47(3):195--209, 2013.

\bibitem{na580}
{S.A.} Nazarov.
\newblock Scattering anomalies in a resonator above thresholds of the
  continuous spectrum.
\newblock {\em Mat. sbornik.}, 206(6):15--48, 2015.
\newblock (English transl.: Sb. Math. 2015. V. 206. N 6. P. 782–-813).

\bibitem{Newt76}
{R.G.} Newton.
\newblock Optical theorem and beyond.
\newblock {\em Am. J. Phys.}, 44(7):639--642, 1976.

\bibitem{NgCH10}
{V.C.} Nguyen, L.~Chen, and K.~Halterman.
\newblock Total transmission and total reflection by zero index metamaterials
  with defects.
\newblock {\em Phys. Rev. Lett.}, 105(23):233908, 2010.

\bibitem{OuMP13}
A.~Ourir, A.~Maurel, and V.~Pagneux.
\newblock Tunneling of electromagnetic energy in multiple connected leads using
  $\varepsilon$-near-zero materials.
\newblock {\em Opt. Lett.}, 38(12):2092--2094, 2013.

\bibitem{PaSV14}
L.~P{\"a}iv{\"a}rinta, M.~Salo, and {E.V.} Vesalainen.
\newblock Strictly convex corners scatter.
\newblock {\em to appear in Rev. Mat. Iberoamericana, arXiv preprint
  arXiv:1404.2513}, 2017.

\bibitem{PAHW07}
F.~Podd, M.~Ali, K.~Horoshenkov, A.~Wood, S.~Tait, J.~Boot, R.~Long, and
  A.~Saul.
\newblock Rapid sonic characterisation of sewer change and obstructions.
\newblock {\em Water Sci. Technol.}, 56(11):131--139, 2007.

\bibitem{RoFi00}
P.~Roux and M.~Fink.
\newblock Time reversal in a waveguide: Study of the temporal and spatial
  focusing.
\newblock {\em J. Acoust. Soc. Am.}, 107(5):2418--2429, 2000.

\bibitem{Rayl71}
{J.W.} Strutt (Lord~Rayleigh).
\newblock On the light from the sky, its polarization and colour.
\newblock {\em Edinb. Dubl. Phil. Mag.}, 41(271):107--120, 1871.

\bibitem{Vain66}
{L.A.} Vainshtein.
\newblock Diffraction theory and the factorization method.
\newblock {\em Sov. Radio, Moscow}, 1966.
\newblock (Russian).

\end{thebibliography}
\bibliographystyle{plain}

\end{document}